\documentclass[a4paper, reqno, 11pt]{amsart}
\usepackage[margin=1.3in]{geometry}
\usepackage{graphicx} 
\usepackage{amsthm,amsfonts,amssymb,amsmath,amsxtra,amsrefs}
\usepackage{mathrsfs}
\usepackage{graphicx}
\usepackage{tikz}
\usepackage{tikz-cd}
\usepackage{xr-hyper}
\usepackage{xr-hyper}
\usepackage[colorlinks=true, citecolor=blue]{hyperref}
\usepackage[noabbrev]{cleveref}
\usepackage{todonotes,cancel}
\usepackage[new]{old-arrows}

\newtheorem{thm}{Theorem}
\newtheorem{theorem}[thm]{Theorem}

\newtheorem{prop}[thm]{Proposition}

\newtheorem{lemma}[thm]{Lemma}

\newtheorem{cor}[thm]{Corollary}

\theoremstyle{definition}

\newtheorem{defi}[thm]{Definition}
\newtheorem{example}[thm]{Example}

\newtheorem{remark}[thm]{Remark}

\numberwithin{thm}{section}

\newcommand{\tUi}{\widetilde{\mathrm{U}}^\imath_n}

\newcommand{\X}{\mathcal{X}}
\newcommand{\OO}{\mathcal{O}}
\newcommand{\T}{\mathbf{T}}
\newcommand{\tT}{\mathbb{T}}
\newcommand{\Qq}{\mathbb{F}}
\newcommand{\tU}{\widetilde{\mathrm{U}}_n}
\newcommand{\ren}[1]{{:#1:}}
\newcommand{\cT}{{T}^{\text{cl}}}
\newcommand{\bT}{T}
\newcommand{\ZtUi}{{_\mathcal{A}\tUi}}
\newcommand{\ZtU}{{_\mathcal{A}\tU}}
\newcommand{\qcom}[2]{[#1,#2]_q}
\newcommand{\oT}{\sigma}
\newcommand{\dU}{\mathrm{U}_n}
\newcommand{\dUi}{{\mathrm{U}^\imath_{n}}}
\newcommand{\tK}{\widetilde{\Upsilon}}
\newcommand{\A}{\mathcal{A}}
\newcommand{\qD}{\mathcal{D}}

\newcommand{\tUio}{\widetilde{\mathrm{U}}_n^{\imath 0}}
\newcommand{\ZdUi}{{_\A \mathrm{U}^\imath_{n}}}
\newcommand{\ZdU}{{_\A \mathrm{U}_n}}
\newcommand{\FOX}{{_{\mathbb{F}}{\mathcal{O}_q(\mathcal{X}_{|\Sigma_n|})}}}
\newcommand{\AT}{_\A\T}

\newcommand\qarrow[2]{\draw[->,shorten >=2pt,shorten <=2pt] (#1) -- (#2) [thick];}
\newcommand\qdarrow[2]{\draw[->,dashed,shorten >=2pt,shorten <=2pt] (#1) -- (#2) [thick];}
\newcommand\qddarrow[2]{\draw[->,double equal sign distance,, shorten >=2pt,shorten <=2pt] (#1) -- (#2) [thick];}

\title{Cluster realisations of $\imath$quantum groups of type AI}
\author[Jinfeng Song]{Jinfeng Song}
\address{Department of Mathematics, National University of Singapore, Singapore.}
\email{j\_song@u.nus.edu}

\begin{document}


\begin{abstract}
    The $\imath$quantum group $\dUi$ of type $\textrm{AI}_n$ is a coideal subalgebra of the quantum group $U_q(\mathfrak{sl}_{n+1})$, associated with the symmetric pair $(\mathfrak{sl}_{n+1},\mathfrak{so}_{n+1})$. In this paper, we give a cluster realisation of the algebra $\dUi$. Under such a realisation, we give cluster interpretations of some fundamental constructions of $\dUi$, including braid group symmetries, the coideal structure, and the action of a Coxeter element. Along the way, we study a (rescaled) integral form of $\dUi$, which is compatible with our cluster realisation. We show that this integral form is invariant under braid group symmetries, and construct PBW-bases for the integral form.
\end{abstract}

\maketitle
\setcounter{tocdepth}{1}
\tableofcontents
\section{Introduction}

Cluster algebras, introduced by Fomin and Zelevinsky \cite{FZ01}, have found exciting connections to many fields of mathematics, and are now an active field of research. The quantum analogue of cluster algebras was studied by Fock--Goncharov \cite{FG09} and Berenstein--Zelevinsky \cite{BZ05}. 

Quantum groups, introduced by Drinfeld and Jimbo, have shown their significance in representation theory, low-dimensional topology, mathematical physics, etc. 

The study of the relations between quantum groups and quantum cluster algebras goes back to the original work of Berenstein and Zelevinsky \cite{BZ05}. Since then, this topic has received many attention and has been extensively studied. 

In the work \cite{SS19}, Schrader and Shapiro gave a cluster algebra realisation of the quantum group $\mathrm{U}$ associated with $\mathfrak{sl}_{n+1}$, that is, they constructed an explicit algebra embedding of $\mathrm{U}$ into a quotient of a quantum cluster algebra. Under the cluster realisation, they obtained cluster interpretations of the quasi $R$-matrix and the coproduct structure of the quantum group. Their constructions are based on the quantum cluster algebra structures associated with the moduli spaces of $\text{PGL}_{n+1}$-local systems on marked surfaces, studied by Fock and Goncharov \cites{FG03,FG09a}.

After establishing the framework to study the quantization of the moduli spaces of $G$-local systems on marked surfaces, Goncharov and Shen \cite{GS22} generalised the construction of Schrader and Shapiro to other finite type Lie algebras. Moreover, they showed Lusztig's braid group symmetries on quantum groups are quasi-cluster automorphisms, that is, they can be expressed as compositions of cluster mutations, permutations of cluster variables, and renormalizations of frozen variables. They also obtained natural Weyl group actions on the corresponding quantum cluster algebras, and conjectured that the images of the cluster realisations are exactly the Weyl group invariants. The conjecture was proved by Shen \cite{Sh22}, in simply laced cases. His proof made an essential use of the quantum cluster duality proved by Davison and Mandel \cite{DM21}.

The notion of \emph{quantum symmetric pairs} was introduced by Letzter \cites{Let99,Let03}. Associated with a symmetric pair $(\mathfrak{g},\mathfrak{g}^\theta)$ of a Lie algebra $\mathfrak{g}$, the quantum symmetric pair $(\mathrm{U},\mathrm{U}^\imath)$ is the pair, consisting of the Drinfeld--Jimbo quantum group $\mathrm{U}$ associated with $\mathfrak{g}$, and a coideal subalgebra $\mathrm{U}^\imath\subset \mathrm{U}$. The subalgebra $\mathrm{U}^\imath$ is referred to as an \emph{$\imath$quantum group}. Following the classification of symmetric pairs of complex simple Lie algebras, the $\imath$quantum groups are classified in terms of Satake diagrams. Quantum groups are $\imath$quantum groups of diagonal types. Therefore the notion of $\imath$quantum groups is a vast generalisation of the notion of quantum groups. Many important properties and applications of quantum groups have found their generalisations in $\imath$quantum groups (see for example \cites{BW18,LW22a,WZ22}). We refer to the survey \cite{Wan21} for the current state of this field.  

The constructions of Schrader--Shapiro and Goncharov--Shen then can be interpreted as the cluster realisations of $\imath$quantum groups of diagonal types $(\mathrm{U}\otimes\mathrm{U},\mathrm{U})$. It is thus natural to expect cluster realisations of $\imath$quantum groups of other types, where fundamental constructions admit cluster interpretations. In the current paper, we construct such a cluster realisation for the $\imath$quantum group of type $\text{AI}_n$, that is, the type associated with the symmetric pair $(\mathfrak{sl}_{n+1},\mathfrak{so}_{n+1})$.

Set $\Qq=\mathbb{Q}(q^{1/2})$ to be the base field. Let $\tU$ be the Drinfeld double quantum group associated with the Lie algebra $\mathfrak{sl}_{n+1}$. It is a Hopf algebra over $\Qq$, with generators $E_i,F_i,K_i,K_i'$, for $1\leq i\leq n$, where $K_iK_i'$ are central elements. Let $\tUi$ be the subalgebra of $\tU$, generated by elements $$B_i=F_i-q^{-1}E_iK_i',\quad k_i=K_iK_i',\quad \text{for } 1\leq i\leq n.$$ Following Lu and Wang \cite{LW22a}, the subalgebra $\tUi$ is called the \emph{universial $\imath$quantum group} of type $\text{AI}_n$. Let us mention that our choice of generators differ by scalars with standard ones (see Remark \ref{rmk:rei}). The Drinfeld--Jimbo quantum group $\dU$ can be recovered by the central reduction $\dU=\tU/(K_iK_i'-1\mid 1\leq i\leq n).$ The $\imath$quantum group $\dUi$ is a subalgebra of $\dU$, which can be obtained by a central reduction of $\tUi$ as well.

Now let us move on to the cluster algebras. Set $\A=\mathbb{Z}[q^{1/2},q^{-1/2}]$. Starting with a quiver $Q$, one can consider the associated \emph{quantum torus algebra} $\T_Q$. This is a non-commutative algebra, where the commuting relations are encoded by the quiver $Q$. The quantum cluster algebra $\mathcal{O}_q(\mathcal{X}_{|Q|})$ is a certain $\A$-subalgebra of $\T_Q$. 

The building block of the cluster realisations of Schrader and Shapiro is the quiver associated with the $n$-triangulation of a triangle. Such a quiver is referred to as a Fock--Goncharov triangle (see Section \ref{sec:clq}). In their work, they considered gluing two Fock--Goncharov triangles along their edges. In our work, we consider gluing one Fock--Goncharov triangle with itself. We call this gluing procedure the \emph{self-amalgamation}.


Let us now state the first main result of the paper.

\begin{thm}(Theorem \ref{thm:emb}, Corollary \ref{cor:fui})\label{thm:main}
    Let $\Sigma_n$ be the quiver obtained by the self-amalgamation of one Fock--Goncharov triangle along its two edges (see \S \ref{sec:quiv}). We explicitly construct a $\mathbb{F}$-algebra embedding of $\tUi$ into the base change of the quantum cluster algebra $\FOX=\mathbb{F}\otimes_\A \mathcal{O}_q(\mathcal{X}_{|\Sigma_n|})$. After the central reduction, we get an injection of the algebra $\dUi$ into a quotient of $\FOX$.
\end{thm}

The embedding in Theorem \ref{thm:main} is not surjective, and it remains open to determine its image. This question seems to be more involved than the similar question in the setting of quantum groups, settled by Shen in simply-laced cased \cite{Sh22}. The main difficulty is that we do not have a geometric framework for our cluster realisations.


Based on the cluster realisation in Theorem \ref{thm:main}, our second result is the cluster interpretation of some fundamental constructions of $\dUi$. To start with, let us recall some basic properties of $\imath$quantum groups.

Firstly, the $\imath$quantum groups are coideal subalgebras, that is, the coproduct of $\dU$ restricts to the map 
$\dUi\rightarrow\dUi\otimes\dU$.

Secondly, the $\imath$quantum groups admit relative braid group symmetries. These symmetries generalise Lusztig's braid group symmetries of quantum groups, and have been widely studied (see \cites{KP11,Dob20,LW22}). In the recent work \cite{WZ22}, Wang and Zhang gave intrinsic constructions of such symmetries. In type AI, the relative braid group coincides with the whole braid group. Therefore we get braid group actions on the algebra $\dUi$. 

Thirdly, the braid group symmetries in type AI also showed in an earlier work of Chekhov \cite{Ch07}, where he observed an additional cyclic symmetry, given by the braid group action of a Coxeter element (see Section \ref{sec:ace}).

All these properties turn out to admit interesting cluster theoretic interpretations. Let us summarise them in the following theorem.

\begin{thm}\label{thm:cra}(Proposition \ref{prop:cop}, Theorem \ref{thm:braid}, Proposition \ref{prop:rho})
    Under the embedding in Theorem \ref{thm:main}, we obtain cluster interpretations of the following fundamental constructions of $\imath$quantum groups.

\begin{itemize}
    \item The coideal structure is interpreted as the amalgamation of two quivers.
    \item Braid group actions are expressed as quasi-cluster automorphisms, that is, they are compositions of cluster mutations, permutations of cluster variables, and renormalizations of frozen variables.
    \item The braid group action associated with a Coxeter element is expressed as the cyclic symmetry of the quiver.
\end{itemize}
\end{thm}

Our third result is on rescaled integral forms of the $\imath$quantum groups.

Since the quantum cluster algebras are originally defined over the ring $\A$, it is natural to study the $\A$-forms of the $\imath$quantum groups which are compatible with the cluster realisations. For quantum groups, Shen \cite{Sh22} constructed such integral forms using the $\A$-spans of the \emph{rescaled} PBW-bases (see \eqref{eq:PBW0}). These rescaled integral forms are different from the Lusztig's integral forms (see \cite{Lu94}*{3.1.13}). To see this, we remark that the specialisations, $q\mapsto 1$, of the rescaled integral forms are commutative, while the specialisations of the Lusztig's forms are not.

Let $\ZdUi$ be the integral form of the $\imath$quantum group, which is induced from the rescaled integral form of the quantum group (see Proposition \ref{prop:inti}). This integral form is different from the integral form studied by Bao and Wang in \cite{BW18}.

Now we state our last result.

\begin{thm}(Corollary \ref{cor:iau}, Corollary \ref{cor:cen})\label{thm:int}
    The $\A$-subalgebra $\ZdUi$ of $\dUi$ satisfies the following properties.
\begin{itemize}
    \item It is invariant under the braid group symmetries of $\imath$quantum groups.
    \item It admits PBW-bases.
    \item It is compatible with the cluster realisation in Theorem \ref{thm:main}, that is, the embedding of $\dUi$ restricts to an $\A$-algebra embedding of $\ZdUi$ into the quotient of the quantum cluster algebra $\mathcal{O}_q(\mathcal{X}_{|\Sigma_n|})$.
\end{itemize}
\end{thm}

Let us mention that, although the first and the second statements seem unrelated to cluster algebras, our proof relies crucially on the cluster realisations of the quantum groups. In fact, the key ingredient to the proof is an expression of the quasi $K$-matrix in terms of quantum dilogarithm (see Lemma \ref{le:Klog}). We believe this expression holds its own interest, and it would also be intriguing to explore analogous expressions for other types. The proof of the last statement makes use of the quasi-cluster nature of the braid group symmetries of $\imath$quantum groups, established in Theorem \ref{thm:cra}.

In \cite{LW22a}, Lu and Wang constructed dual PBW-bases for quasi-split $\imath$quantum groups, based on their Hall algebra realisations of $\imath$quantum groups. It would be interesting to see the connections with their work.  

Our constructions are inspired by results in Poisson geometry. It is known that the quantum groups are quantization of the dual Poisson group $G^*$. From this point of view, the $\imath$quantum groups are quantization of certain Poisson homogeneous spaces of $G^*$. In type AI, the corresponding Poisson homogeneous spaces turn out to be the spaces of the unipotent upper-triangular matrices, equipped with the Poisson structures introduced by Dubrovin \cite{Dub96} (see also \cite{Ug99}). In the recent work \cite{CS23}, Chekhov and Shapiro constructed log-canonical coordinates for such Poisson spaces (see also \cite{BT22} from an analytic point of view). Our cluster realisations are highly motivated by the constructions of Chekhov and Shapiro.

The precise connections between $\imath$quantum groups of general types and Poisson homogeneous spaces will be treated in a separate paper. On the other hand, constructing cluster realisations for general type $\imath$quantum groups seems to be more challenging. 

The paper is organised as follows. In Section \ref{sec:2}, we collect known facts on quantum cluster algebras, quantum groups and quantum symmetric pairs. Section \ref{sec:4} is devoted to study the rescaled integral forms of $\imath$quantum groups. In \ref{sec:pbw}, we construct PBW-bases for $\imath$quantum groups. In \ref{sec:qK}, we relate quasi $K$-matrices with quantum dilogarithm. In \ref{sec:intf}, we construct integral forms of $\imath$quantum groups and show that they are invariant under braid group symmetries. Section \ref{sec:5} is the main part of the paper. In \ref{sec:albe}, we construct an explicit algebra embedding of $\tUi$ into the base change of a quantum cluster algebra. In \ref{sec:coid}, we interpret the coideal structure in terms of the amalgamation of two quivers. In \ref{sec:qucl}, we show that braid group actions on $\imath$quantum groups are quasi-cluster, and as a consequence, we show that the rescaled integral forms of $\imath$quantum groups are compatible with our cluster realisations. In \ref{sec:rho}, we explain that the braid group action associated with a Coxeter element corresponds to the cyclic symmetry of the quiver. 

\vspace{.2cm}\noindent {\bf Acknowledgement: }The author is supported by Huanchen Bao’s MOE grant A-0004586-00-00 and A-0004586-01-00.

\section{Preliminaries}\label{sec:2}

\subsection{Quantum cluster algebras}\label{sec:defcl}

We recall basic facts of quantum cluster algebras in this subsection. 

A \emph{seed} consists of a triple $\Sigma= (I,I_0,\varepsilon)$, where $I$ is a finite set, $I_0\subset I$ is a subset, and $\varepsilon=(\varepsilon_{ij})_{i,j\in I}$ is a skew-symmetric $\mathbb{Z}/2$-valued matrix, such that $\varepsilon_{ij}\in\mathbb{Z}$ unless $i,j\in I_0$. 

For each $k\in I\backslash I_0$, the \emph{seed mutation} in the direction $k$ produces a new seed $\mu_k \Sigma=(I,I_0,\varepsilon')$ where
\begin{equation*}
    \varepsilon'_{ij}=\begin{cases}
        -\varepsilon_{ij}, & \text{if} \; k\in\{i,j\}, \\
        \varepsilon_{ij}, & \text{if} \; \varepsilon_{ik}\varepsilon_{kj}\leq0,\; k\notin\{i,j\},\\
        \varepsilon_{ij}+|\varepsilon_{ik}|\varepsilon_{kj}, & \text{if} \; \varepsilon_{ik}\varepsilon_{kj}>0, \; k\notin\{i,j\}.
    \end{cases}
\end{equation*}

One can encode the seed $\Sigma=(I,I_0,\varepsilon)$ by a quiver, with vertices labelled by elements in $I$, and with the adjacency matrix $\varepsilon$. Vertices in $I_0$ are called \emph{frozen vertices} and are usually denoted by boxes in the quiver. The mutation $\mu_k$ of the corresponding quiver can be performed in three steps:

1) reverse all the arrows incident to the vertex $k$;

2) for each pair of arrows $k\rightarrow i$ and $j\rightarrow k$ draw an arrow $i\rightarrow j$;

3) delete pairs of arrows $i\rightarrow j$ and $j\rightarrow i$ going in the opposite directions.

\begin{figure}[h]
\begin{tikzpicture}
\begin{scope}[>=latex]
 \node[circle,draw] (C) at (0,0) {\tiny{1}};
 \node[circle,draw] (D) at (2,0) {\tiny{2}};
 \node[rectangle,draw] (A) at (4,0.8) {\tiny{3}};
 \node[rectangle,draw] (B) at (4,-0.8) {\tiny{4}};
\qddarrow{D}{C};
\qarrow{D}{A};
\qarrow{B}{D};
\qdarrow{A}{B};
\qarrow{C}{B}
\end{scope}
\draw[thick,|->] (5,0) -- (6,0);
\node at (5.5,0.5) {$\mu_2$};
\begin{scope}[>=latex, xshift=7cm]
 \node[circle,draw] (C) at (0,0) {\tiny{1}};
 \node[circle,draw] (D) at (2,0) {\tiny{2}};
 \node[rectangle,draw] (A) at (4,0.8) {\tiny{3}};
 \node[rectangle,draw] (B) at (4,-0.8) {\tiny{4}};
\qddarrow{C}{D};
\qarrow{A}{D};
\qarrow{D}{B};
\qdarrow{B}{A};
\qarrow{B}{C};
\end{scope}
\end{tikzpicture}
\caption{A quiver mutation in the direction 2. Dashed arrows between frozen vertices denote arrows of weight $\frac{1}{2}$.}
\end{figure}
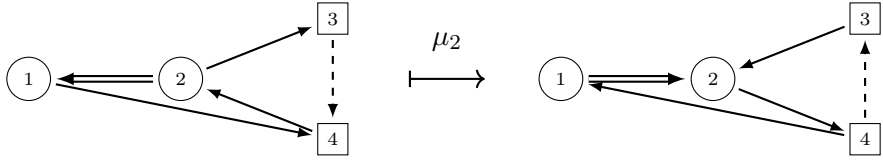

Take a quiver $\Sigma=(I,I_0,\varepsilon)$. Set $\Qq=\mathbb{Q}(q^{1/2})$. Define the \emph{quantum torus algebra} $\T_\Sigma$ to be the unital non-commutative $\Qq$-algebra, which is generated by $X_i^{\pm1}$, for $i\in I$, and subjects to relations
$$ X_iX_j=q^{2\varepsilon_{ji}}X_jX_i\qquad \text{for }i,j\in I.$$ Denote by $\tT_\Sigma$ the non-commutative fraction field of $\T_\Sigma$.

Set $\A=\mathbb{Z}[q^{1/2},q^{-1/2}]$. The $\Qq$-algebra $\T_\Sigma$ has a natural $\A$-form $\AT_\Sigma$, which is the $\A$-subalgebra of $\T_\Sigma$, generated by $X_i^{\pm1}$, for $i\in I$. 

For a quiver $\Sigma$, there is a nature $*$-algebra structure on $\T_{\Sigma}$. Namely, there is an anti-automorphism $*$ on $\T_{\Sigma}$, which preserves all the generators $X_i$, and maps $q$ to $q^{-1}$. For any Laurent monomial $$M=\prod_{i\in v(\Sigma)}X_i^{a_i}\qquad \text{in }\T_{\Sigma},$$ where $v(\Sigma)$ denotes the set of vertices of $\Sigma$, $a_i$ are integers, and the product is taken in a certain order, there exists a unique element $\ren{M}$ in $\tT_{\Sigma}$, such that, 

1) it equals $M$ up to a scalar in $q^{\mathbb{Z}}$; 

2) it is fixed by the automorphism $*$. 

The element $\ren{M}$ is called the \emph{renormalized monomial} of $M$.

For $k\in I\backslash I_0$, write $\Sigma'=\mu_k\Sigma$. The \emph{cluster mutation} in the direction $k$ is an isomorphism of skew fields $\mu_k: \tT_{\Sigma'}\rightarrow\tT_\Sigma$, given by
\begin{equation*}
    \mu_k(X_i)=\begin{cases}
        X_k^{-1}, & \text{if}\;i=k,\\
        X_i\prod_{r=1}^{\varepsilon_{ki}}(1+q^{2r-1}X_k^{-1})^{-1}, & \text{if}\; i\neq k\text{ and }\varepsilon_{ki}\geq 0,\\
        X_i\prod_{r=1}^{-\varepsilon_{ki}}(1+q^{2r-1}X_k), & \text{if}\;i\neq k\text{ and }\varepsilon_{ki}\leq 0.
    \end{cases}
\end{equation*}

The \emph{quantum cluster algebra} $\OO_q(\X_{|\Sigma|})$ associated with the quiver $\Sigma$ is the $\A$-subalgebra of $\AT_\Sigma$, which consists of elements that remain to be Laurent polynomials after any finite sequence of cluster mutations.

Let us recall the procedure of \emph{quiver amalgamation} of two quivers by a subset of frozen variables, following \cite{GS22} (see also \cite{SS19}). Let $Q_1$, $Q_2$ be two quivers, with certain subsets $I_1$, $I_2$ of frozen variables in $Q_1$, $Q_2$, respectively. Suppose there is a bijection $\phi:I_1\rightarrow I_2$. We can then amalgamate quivers $Q_1$ and $Q_2$ by subsets $I_1$ and $I_2$, along the map $\phi$. The resulting quiver $Q$ is obtained in the following two steps:

1) for any $i\in I_1$, identify vertices $v_i\in Q_1$ and $v_{\phi(i)}\in Q_2$ in the union $Q_1\sqcup Q_2$;

2) for any pair $i,j\in I_1$ with an arrow $v_i\rightarrow v_j$ in $Q_1$ of weight $\varepsilon_{ij}$ and an arrow $v_{\phi(i)}\rightarrow v_{\phi(j)}$ in $Q_2$ of weight $\varepsilon_{\phi(i),\phi(j)}$, the weight between corresponding vertices in $Q$ is assumed to be $\varepsilon_{ij}+\varepsilon_{\phi(i),\phi(j)}$.

Then there is an algebra embedding $\T_{Q}\rightarrow\T_{Q_1}\otimes\T_{Q_2}$ of the corresponding quantum torus algebras, given by
\begin{equation}\label{eq:aml}
    X_i\mapsto\begin{cases}
        X_i\otimes 1, & \text{for }i\in Q_1\backslash I_1,\\
        1\otimes X_i, & \text{for }i\in Q_2\backslash I_2,\\
        X_i\otimes X_{\phi(i)}, & \text{otherwise.}
    \end{cases}
\end{equation}

In simple words, the quiver amalgamation is nothing but gluing two quivers by identifying certain frozen variables. 

We define a similar procedure, called the \emph{self-amalgamation} of a quiver $Q$. The only difference with the previous definition is that we assume $Q_1=Q_2=Q$, and $I_1$, $I_2$ are two disjoint subsets of frozen vertices of $Q$. The frozen variables in $I_1$ (or equivalently, in $I_2$) are assumed to be unfrozen after the self-amalgamation.

\subsection{Quantum groups}\label{sec:3} 
Fix an integer $n\geq 1$. In what follows, we consider the Lie algebra $\mathfrak{g}=\mathfrak{sl}_{n+1}(\mathbb{C})$. Let $\mathcal{R}=\mathcal{R_+}\sqcup\mathcal{R_-}$ be the root system of $\mathfrak{g}$, with a partition into positive roots and negative roots. Let $\Pi=\{\alpha_i\mid 1\leq i\leq n\}\subset \mathcal{R}_+$ be the set of simple roots. Denote the root lattice by $Q=\oplus_{i=1}^n\mathbb{Z}\alpha_i$. For $\lambda,\mu\in Q$, we write $\lambda\preceq \mu$ if $\mu-\lambda$ is a positive linear combination of simple roots. Then $\preceq$ defines a partial order on the root lattice. The Weyl group $W=S_{n+1}$ of $\mathfrak{g}$ acts on the root lattice $Q$. The simple reflections are denoted by $s_i$, for $1\leq i\leq n$. We denote by $w_0$ the longest element in $W$, and write $r=n(n+1)/2$ to be the number of positive roots.

\begin{defi}\label{def:tU}
    The algebra $\tU$ is the unital $\Qq$-algebra generated by $E_i$, $F_i$, $K_i^{\pm1}$, $K_i'^{\pm1}$, for $1\leq i\leq n$, which subject to the following relations for $1\leq i,j\leq n$,
    \begin{align*}
        & K_iK_i^{-1}=K_i^{-1}K_i=K_i'K_i'^{-1}=K_i'^{-1}K_i'=1,\qquad [K_i,K_j]=[K_i',K_j']=0,\\
        & K_iE_j=q^{a_{ij}}E_jK_i,\qquad K_i'E_j=q^{-a_{ij}}E_jK_i',\\
        & K_iF_j=q^{-a_{ij}}F_jK_i,\qquad K_i'F_j=q^{a_{ij}}F_jK_i',\\
        & [E_i,F_j]=\delta_{ij}(q-q^{-1})(K_i'-K_i),\\
        & E_i^2E_j-(q+q^{-1})E_iE_jE_i+E_jE_i^2=0,\quad\text{if }|i-j|=1,\\
        & F_i^2F_j-(q+q^{-1})F_iF_jF_i+F_jF_i^2=0,\quad\text{if }|i-j|=1,\\
        &[E_i,E_j]=[F_i,F_j]=0,\quad\text{if }|i-j|>1.
    \end{align*}
Here $a_{ij}$ is the entry in the Cartan matrix, which equals $2$ if $i=j$, equals $-1$ if $|i-j|=1$, and equals $0$ if $|i-j|>1$.
\end{defi}

The algebra $\tU$ is called the Drinfeld double quantum group associated with the Lie algebra $\mathfrak{g}$. It is a Hopf algebra, with the coproduct given by
\begin{equation*}
\begin{array}{ll}
    \Delta(E_i)=E_i\otimes 1+K_i\otimes E_i,\qquad & \Delta(K_i)=K_i\otimes K_i, \\
    \Delta(F_i)=F_i\otimes K_i'+1\otimes F_i, \qquad & \Delta(K_i')=K_i'\otimes K_i',
\end{array}
\end{equation*}
for $1\leq i\leq n$.

The algebra $\tU$ is naturally a $Q$-graded algebra, where
\begin{equation}\label{eq:grad}
    \text{deg}(F_i)=\alpha_i,\;\text{deg}(E_i)=-\alpha_i,\;\text{deg}(K_i)=\text{deg}(K_i')=0,\; \text{for }1\leq i\leq n.
\end{equation}

Let $\tU^{>}$ (resp., $\tU^{<}$) be the unital subalgebra of $\tU$ generated by $E_i$ (resp., $F_i$), for $1\leq i\leq n$. Let $\tU^0$ be the unital subalgebra of $\tU$ generated by $K_i$ and $K_i'$, for $1\leq i\leq n$.  

One has the \emph{triangular decomposition} of $\tU$, which is the $\Qq$-vector space isomorphism
\begin{equation}\label{eq:tria}
    m: \tU^{<}\otimes \tU^{0}\otimes \tU^>\overset{\sim}{\longrightarrow}\tU,
\end{equation}
given by the multiplication. Here the tensor products are taken over $\Qq$.

The Drinfeld--Jimbo quantum group $\dU$ is the unital $\Qq$-algebra with generators $E_i$, $F_i$, $K_i^{\pm1}$, for $1\leq i\leq n$, and with the relations obtained by replacing $K_i'$ by $K_i^{-1}$ in  Definition \ref{def:tU}. 

There is a surjective algebra homomorphism 
\begin{equation}
    \pi:\tU\longrightarrow\dU,
\end{equation}
called the \emph{central reduction map} of $\tU$, given by
\[
E_i\mapsto E_i,\;F_i\mapsto F_i,\;K_i\mapsto K_i,\;K_i'\mapsto K_i^{-1},\quad\text{for }1\leq i\leq n.
\]

\begin{remark}\label{rmk:res}
Our generators for (universal) quantum groups differ by scalars with the commonly used ones. Let $\mathbf{E}_i$, $\mathbf{F}_i$, $\mathbf{K}_i$, $\mathbf{K}_i'$ ($1\leq i\leq n$) be the generators of the $\mathbb{Q}(q)$-algebra $\widetilde{\mathbf{U}}$ as in \cite{WZ22}*{\S 2.1}, associated with $\mathfrak{g}$. Then one has the $\Qq$-algebra isomorphism:
\[
\tU\overset{\sim}{\longrightarrow}\Qq\otimes_{\mathbb{Q}(q)}\widetilde{\mathbf{U}},
\]
given by for $1 \leq i\leq n$,
\begin{equation}\label{eq:resc0}
E_i\mapsto q^{1/2}(q^{-1}-q)\mathbf{E}_i,\;F_i\mapsto q^{-1/2}(q-q^{-1})\mathbf{F}_i,\;K_i\mapsto\mathbf{K}_i,\;K_i'\mapsto\mathbf{K}_i'.
\end{equation}
\end{remark}
\
\subsection{$\imath$Quantum groups}\label{sec:iqg}

The $\imath$quantum groups are certain coideal subalgebras of quantum groups.

\begin{defi}
    The algebra $\tUi$ is the unital $\Qq$-subalgebra of $\tU$,  generated by elements $B_i$, $k_i^{\pm1}$, for $1\leq i\leq n$, where
    \begin{equation*}
        B_i=F_i-q^{-1}E_iK_i',\quad k_i=K_iK_i'.
    \end{equation*}
    Following \cite{LW22}, The algebra $\tUi$ is called the \emph{universal $\imath$quantum group} of type $\mathrm{AI}_n$. 
\end{defi}

Set $\tUio$ to be the unital $\Qq$-subalgebra of $\tUi$, generated by $k_i^{\pm1}$, for $1\leq i\leq n$. Then $\tUio$ is a central subalgebra of $\tUi$.

The algebra $\tUi$ is a right coideal subalgebra of $\tU$, that is, the coproduct of $\tU$ restricts to
\begin{equation}\label{eq:cop}
    \Delta: \tUi\longrightarrow \tUi\otimes \tU.
\end{equation}

One has the following presentation of $\tUi$. 

\begin{prop}(cf. \cite{LW22a}*{Proposition 6.4})
    The algebra $\tUi$ is the unital $\Qq$-algebra generated by $B_i$, $k_i^{\pm1}$, for $1\leq i\leq n$, with relations
    \begin{align}
        & k_ik_i^{-1}=k_i^{-1}k_i=1, \qquad [k_i,k_j]=[k_i,B_j]=0,\label{eq:R0}\\
        &[B_i,B_j]=0,\quad\text{if }|i-j|>1, \label{eq:R1}\\
        &B_jB_i^2-(q+q^{-1})B_iB_jB_i+B_i^2B_j=(q-q^{-1})^2B_jk_i, \quad \text{if }|i-j|=1. \label{eq:R2}
    \end{align} 
\end{prop}

\begin{remark}\label{rmk:rei}
    Let $\mathbf{B}_i$, $\widetilde{\mathbf{k}}_i$ ($1\leq i\leq n$) be the generators of the $\mathbb{Q}(q)$-algebra $\widetilde{\mathbf{U}}^\imath$ in \cite{WZ22}*{\S 2.4} associated with the Satake diagram of type $\mathrm{AI}_n$. Then one has the $\Qq$-algebra isomorphism
$$\tUi\overset{\sim}{\longrightarrow}\Qq\otimes_{\mathbb{Q}(q)}\widetilde{\mathbf{U}}^\imath,$$
given by 
\begin{equation}\label{eq:res1}
B_i\mapsto q^{-1/2}(q-q^{-1})\mathbf{B}_i,\quad k_i\mapsto \widetilde{\mathbf{k}}_i,\quad \text{for }1\leq i\leq n.
\end{equation}
\end{remark}

The $\imath$quantum group $\dUi$ is defined to be the $\Qq$-subalgebra of $\dU$, generated by elements
\begin{equation}
    B_i=F_i+q^{-1} E_iK_i^{-1},\quad \text{for }1\leq i\leq n.
\end{equation}
With our rescaling, the algebra $\dUi$ is (the field extension of) the $\imath$quantum group of type $\mathrm{AI}_n$ as in \cite{WZ22}*{\S 2.5}, associated with the parameters $\varsigma_i=-1$, for $1\leq i\leq n$.

One may also consider $\imath$quantum groups associated with other parameters $\boldsymbol{\varsigma}=(\varsigma_i)_{i=1}^n\in\Qq^n$ (cf. \emph{loc.cit.}). Thanks to \cite{Wa21}*{Lemma 2.5.1}, $\imath$quantum groups associated with different parameters are isomorphic after field extensions.

There is an algebra surjection
\begin{equation}\label{eq:icen}
    \pi^\imath:\tUi\longrightarrow \dUi,
\end{equation}
which is called the \emph{central reduction} of $\tUi$, given by
\begin{equation}
    B_i\mapsto B_i,\quad k_i\mapsto -1,\quad \text{for }1\leq i\leq n.
\end{equation}
The kernel of $\pi^\imath$ is the two-sided ideal generated by $k_i+1$, for $1\leq i\leq n$. 

Note that $\pi^\imath$ is not the restriction of $\pi$. They only coincide after a twist on generators (see the proof of Proposition \ref{prop:inti}).

\subsection{Braid group symmetries}\label{sec:braid}

For $1\leq i\leq n$, there is a $\Qq$-algebra automorphism $\oT_i$ of $\tU$, given by
\begin{equation*}
    E_j\mapsto\begin{cases}
        qK_i'^{-1}F_i, &\text{if }i=j,\\
        \frac{q^{1/2}E_iE_j-q^{-1/2}E_jE_i}{q-q^{-1}}, &\text{if }|i-j|=1,\\
        E_j, &\text{if }|i-j|>1,
    \end{cases}
    \quad K_j\mapsto\begin{cases}
        K_i^{-1}, &\text{if }i=j,\\
        K_iK_j, &\text{if }|i-j|=1,\\
        K_j, &\text{if }|i-j|>1,
    \end{cases}
\end{equation*}
and
\begin{equation*}
    F_j\mapsto\begin{cases}
        q^{-1}E_iK_i^{-1}, &\text{if }i=j,\\
        \frac{q^{1/2}F_iF_j-q^{-1/2}F_jF_i}{q-q^{-1}}, &\text{if }|i-j|=1,\\
        F_j, &\text{if }|i-j|>1,
    \end{cases}
    \quad K_j'\mapsto\begin{cases}
        K_i'^{-1}, &\text{if }i=j,\\
        K_i'K_j', &\text{if }|i-j|=1,\\
        K_j', &\text{if }|i-j|>1.
    \end{cases}
\end{equation*}
The automorphisms $\{\sigma_i\mid 1\leq i\leq n\}$ satisfies the braid group relation, that is, $\sigma_i\sigma_j\sigma_i=\sigma_j\sigma_i\sigma_j$ for $|i-j|=1$, and $\sigma_i\sigma_j=\sigma_j\sigma_i$ for $|i-j|>1$.

We remark that the automorphism $\oT_i$ is the $\widetilde{T}_{i,+1}'$ in \cite{WZ22}*{\S 2.2} under the rescaling \eqref{eq:resc0}, and is the inverse of the braid group action associated with $i$ in \cite{Sh22}.

Each $\sigma_i$ preserves the ideal of $\tU$ generated by $K_jK_j'-1$, for $1\leq j\leq n$. Therefore it descends to an algebra automorphism of $\dU$.

Take a reduced expression $\mathbf{i}=(i_1,\cdots,i_r)$ of $w_0$. Given $\mathbf{a}=(a_1,\cdots,a_r)$ and $\mathbf{b}=(b_1,\cdots,b_r)$ in $\mathbb{N}^r$, set 
\begin{equation*}
\begin{split}
    &F_{\mathbf{i}}(\mathbf{a})=E_{i_1}^{a_1}\oT_{i_1}(E_{i_2}^{a_2})\cdots \oT_{i_1}\cdots\oT_{i_{r-1}}(E_{i_r}^{a_r}),\quad\text{and }\\
    &E_{\mathbf{i}}(\mathbf{b})=F_{i_1}^{b_1}\oT_{i_1}(F_{i_2}^{b_2})\cdots \oT_{i_1}\cdots\oT_{i_{r-1}}(F_{i_r}^{b_r}).
\end{split}
\end{equation*}
For $\mathbf{c}=(c_1,\cdots,c_n)$ and $\mathbf{d}=(d_1,\cdots,d_n)$ in $\mathbb{Z}^n$, we set
\begin{equation*}
    K(\mathbf{c})=K_1^{c_1}\cdots K_n^{c_n}\quad \text{and}\quad K'(\mathbf{d})=K'^{d_1}_1\cdots K'^{d_n}_n.
\end{equation*}
Then the set
\begin{equation}\label{eq:PBW0}
    \{F_{\mathbf{i}}(\mathbf{a})K(\mathbf{c})K'(\mathbf{d})E_{\mathbf{i}}(\mathbf{b})\mid \mathbf{a},\mathbf{b}\in\mathbb{N}^r,\mathbf{c},\mathbf{d}\in\mathbb{Z}^n\}
\end{equation}
gives a $\Qq$-basis of $\tU$, which is called the (rescaled) PBW-basis of $\tU$ associated with $\mathbf{i}$. 

Recall $\A=\mathbb{Z}[q^{1/2},q^{-1/2}]$. Let $\ZtU$ be the $\A$-submodule of $\tU$ spanned by PBW-basis associated with $\mathbf{i}$. By \cite{Sh22}*{Theorem 5.13}, the $\A$-submodule $\ZtU$ is moreover an $\A$-subalgebra, which is independent of the choice of $\mathbf{i}$.

Let $\ZdU=\pi(\ZtU)$ be the $\A$-subalgebra of $\dU$. It has an $\A$-linear basis
\begin{equation}
    \{F_{\mathbf{i}}(\mathbf{a})K(\mathbf{c})E_{\mathbf{i}}(\mathbf{b})\mid \mathbf{a},\mathbf{b}\in\mathbb{N}^r,\mathbf{c}\in\mathbb{Z}^n\}.
\end{equation}
This basis is referred to as the PBW-basis of $\dU$ associated with $\mathbf{i}$.

We next recall the (relative) braid group symmetries on $\imath$quantum groups following \cite{WZ22}. Let us mention that the braid group symmetries on quantum groups do not preserve the $\imath$quantum groups in general. Therefore the braid group symmetries on $\imath$quantum groups are not the restriction the braid group symmetries of quantum groups.

By \cite{WZ22}*{Theorem C} (with our rescaling), there is a $\Qq$-algebra automorphism $T_i$ of $\tUi$, for $1\leq i\leq n$, given by
\begin{equation}\label{braid}
    B_j\mapsto\begin{cases}
        -k_i^{-1}B_i, &\text{if }i=j,\\
        \frac{q^{1/2}B_iB_j-q^{-1/2}B_jB_i}{q-q^{-1}}, &\text{if }|i-j|=1,\\
        B_j, & \text{if }|i-j|>1, 
    \end{cases}
    \quad 
    k_j\mapsto\begin{cases}
        k_i^{-1} &\text{if }i=j,\\
        -k_ik_j &\text{if }|i-j|=1,\\
        k_j &\text{if }|i-j|>1.
    \end{cases}
\end{equation}
The automorphism $\bT_i$ is denoted by $\widetilde{\mathbf{T}}_{i,+1}'$ in \emph{loc.cit.}

Note that each $\bT_i$ preserves the ideal generated by $k_j+1$, for $1\leq j\leq n$. Therefore $\bT_i$ descends to an algebra automorphism of $\dUi$.

\subsection{The action of a Coexter element}\label{sec:ace}

One interesting feature of the braid group action in type AI is that, the action related to the Coxeter element $s_1s_2\cdots s_n$ acts via a cyclic symmetry.

Let us denote $c=s_1s_2\cdots s_n$ in $W$, and write $\bT_c=\bT_1\bT_2\cdots\bT_n$ to be the automorphism of $\tUi$.

\begin{prop}\label{prop:tci}(cf. \cite{Ch07}*{Lemma 6})
    As an automorphism on $\tUi$, we have 
    \begin{equation}\label{eq:cyc}
        \bT_c^{n+1}=id.
    \end{equation}
\end{prop}

\begin{proof}
    Set $B_{n+1}=\bT_c(B_n)$. The proposition will follow if we can show
    \begin{equation}\label{eq:Tcy}
        \bT_c(B_i)=B_{i+1},\quad \text{for }1\leq i\leq n+1,
    \end{equation}
    where the subscript is considered modular $n+1$.

    For $1\leq i<n$, it is easy to see that $c(\alpha_i)=\alpha_{i+1}$. Then by \cite{WZ22}*{Theorem 7.13}, we have $\bT_c(B_i)=B_{i+1}$.

    For $i=n$, the equation \eqref{eq:Tcy} follows from the definition. For $i=n+1$, we need to show $\bT_c^2(B_n)=B_1$. Notice that $s_1c^2s_{n}=s_2s_3\cdots s_ns_1s_2\cdots s_{n-1}$, where the right hand side is a reduced expression, and $s_1c^2s_{n}(\alpha_n)=\alpha_1$. Therefore by \emph{loc.cit.}, we have
    \begin{equation*}
        \bT_2\bT_3\cdots\bT_n\bT_1\bT_2\cdots\bT_{n-1}(B_n)=B_1.
    \end{equation*}
    So we get
    \begin{equation*}
        \bT_c^2(B_n)=\bT_1\cdot\bT_2\bT_3\cdots\bT_n\bT_1\bT_2\cdots\bT_{n-1}(-k_n^{-1}B_n)=\bT_1(-k_1^{-1}B_1)=B_1,
    \end{equation*}
    which completes the proof.
\end{proof}


\subsection{Quasi $K$-matrices}\label{sec:K}
By \cite{WZ22}, braid group symmetries on quantum groups and relative braid group symmetries on $\imath$quantum groups are related by the rank-one \emph{quasi K-matrices}. Let us recall their results in our setting.

Recall the grading of $\tU$ from \eqref{eq:grad}. Let $Q_-=\sum_{i=1}^n\mathbb{Z}_{\leq 0}\alpha_i$ be the submonoid of the root lattice $Q$, consisting of non-positive combinations of simple roots. For $\mu\in Q_-$, let $\widetilde{\mathrm{U}}_{n,\mu}^{>0}$ be the subspace of $\tU^{>0}$ consisting of homogeneous elements of degree $\mu$. Let $\widehat{\mathrm{U}}^{>0}_n$ be the completion of the algebra $\tU^{>0}$, which consists of formal sums $f=\sum_{\mu\in Q_-}f_\mu$ where $f_\mu\in \widetilde{\mathrm{U}}^{>0}_{n,\mu}$. 


For each $1\leq i\leq n$, there is a unique element $\tK_i=\sum_{n\geq 0}a_{2n}E_i^{2n}$ in $\widehat{\mathrm{U}}^{>0}_n$, such that, $a_{2n}\in\Qq$, $a_0=1$ and the equality
\begin{equation}\label{eq:it}
    B_i\tK_i=\tK_i(F_i-qE_iK_i)
\end{equation}
holds in a completion of $\tU$. 


By \cite{WZ22}*{Theorem B}, one has 
\begin{equation}\label{eq:intw}
    \oT_i^{-1}(x)\tK_i=\tK_i\bT^{-1}_i(x),\quad \text{for }x\in \tUi,
\end{equation}
in a certain completion algebra.

\subsection{Cluster realisations of quantum groups}\label{sec:clq}

In this subsection, we recall the cluster realisation of the algebra $\tU$ in \cite{SS19}. 

We start by describing the construction of the quiver $\widetilde{\Sigma_n}$, associated with the $n$-triangulation of a triangle. Take an equilateral triangle $ABC$. For each edge of $ABC$, we draw $n$ equidistant lines inside the triangle which are parallel to that edge. In this way we get $(n+1)^2$ small triangles inside $ABC$. The resulting picture is called an \emph{$n$-triangulation} of the triangle $ABC$. To obtain a quiver from this picture, we put a vertex on each vertex of small triangles, except for vertices $A$, $B$ and $C$. The vertices on edges $AB$, $BC$ and $CA$ are designed to be frozen. For vertices lying on the same side of the triangle $ABC$ we put dashed arrows in the direction $\overrightarrow{AB}$, $\overrightarrow{BC}$ and $\overrightarrow{CA}$. We put an arrow for each edge of small triangles which does not lie on the edge of $ABC$, and the direction is chosen in such a way that inside each small triangle we get a loop. The resulting quiver is denoted by $\widetilde{\Sigma_n}$. 

Let $\qD_n$ be the quiver obtained by amalgamating two $\widetilde{\Sigma_n}$-quivers by identifying corresponding frozen vertices on two sides of each triangle. 

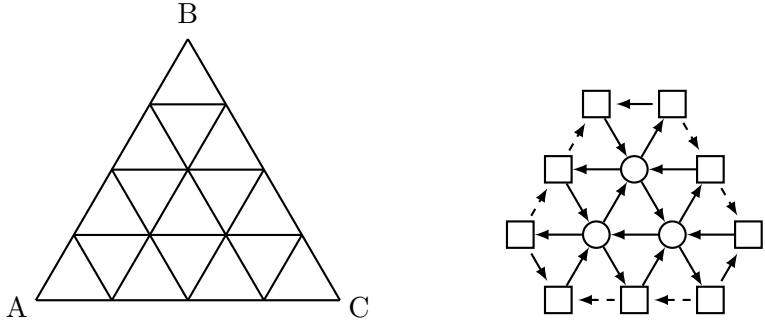
\begin{figure}[h]
\centering
\begin{tikzpicture}[thick, y=0.866cm, x=0.5cm]

\node at (-0.5,-0.1) {A};
\node at (4,4.4) {B};
\node at (8.5,-0.1) {C};

\draw (0,0) to (8,0);
\draw (1,1) to (7,1);
\draw (2,2) to (6,2);
\draw (3,3) to (5,3);

\draw (0,0) to (4,4);
\draw (2,0) to (5,3);
\draw (4,0) to (6,2);
\draw (6,0) to (7,1);

\draw (4,4) to (8,0);
\draw (3,3) to (6,0);
\draw (2,2) to (4,0);
\draw (1,1) to (2,0);

\end{tikzpicture}
\qquad\qquad
\begin{tikzpicture}[every node/.style={inner sep=0, minimum size=0.35cm, thick}, thick, y=0.866cm, x=0.5cm]
\begin{scope}[>=latex]
\node at (2,-0.2) {};

\node (1) at (2,0) [draw] {};
\node (2) at (4,0) [draw] {};
\node (3) at (6,0) [draw] {};
\node (4) at (1,1) [draw] {};
\node (5) at (3,1) [circle, draw] {};
\node (6) at (5,1) [circle, draw] {};
\node (7) at (7,1) [draw] {};
\node (8) at (2,2) [draw] {};
\node (9) at (4,2) [circle, draw] {};
\node (10) at (6,2) [draw] {};
\node (11) at (3,3) [draw] {};
\node (12) at (5,3) [draw] {};

\draw[->, dashed, shorten >=2pt,shorten <=2pt] (2) to (1);
\draw[->, dashed, shorten >=2pt,shorten <=2pt] (3) to (2);
\draw[->, shorten >=2pt,shorten <=2pt] (3) to (7);
\draw[->, dashed, shorten >=2pt,shorten <=2pt] (10) to (7);
\draw[->, dashed, shorten >=2pt,shorten <=2pt] (12) to (10);
\draw[->, shorten >=2pt,shorten <=2pt] (12) to (11);
\draw[->, dashed, shorten >=2pt,shorten <=2pt] (8) to (11);
\draw[->, dashed, shorten >=2pt,shorten <=2pt] (4) to (8);
\draw[->, shorten >=2pt,shorten <=2pt] (4) to (1);

\draw[->] (9) to (12);
\draw[->] (5) to (9);
\draw[->] (1) to (5);
\draw[->] (6) to (10);
\draw[->] (2) to (6);

\draw[->] (6) to (3);
\draw[->] (9) to (6);
\draw[->] (11) to (9);
\draw[->] (5) to (2);
\draw[->] (8) to (5);

\draw[->] (7) to (6);
\draw[->] (6) to (5);
\draw[->] (5) to (4);
\draw[->] (10) to (9);
\draw[->] (9) to (8);
\end{scope}
\end{tikzpicture}
\caption{3-triangulation of the triangle ABC and the $\widetilde{\Sigma_3}$-quiver}

\end{figure}

\begin{figure}[h]
\centering
\begin{tikzpicture}[every node/.style={inner sep=0, minimum size=0.35cm, thick}, thick, x=0.75cm, y=0.5cm]

\node at (-3.7,-2) {$V_3$};
\node at (-3.7, 0) {$V_2$};
\node at (-3.7, 2) {$V_1$};
\node at (3.7, -2) {$\Lambda_1$};
\node at (3.7, 0) {$\Lambda_2$};
\node at (3.7, 2) {$\Lambda_3$};

\node (1) at (-3,-2) [draw] {};
\node (2) at (0,-5) [draw,circle] {};
\node (3) at (3,-2) [draw] {};
\node (4) at (-3,0) [draw] {};
\node (5) at (-2,-1) [draw,circle] {};
\node (6) at (0,-3) [draw,circle] {};
\node (7) at (2,-1) [draw,circle] {};
\node (8) at (3,0) [draw] {};
\node (9) at (-3,2) [draw] {};
\node (10) at (-2,1) [draw,circle] {};
\node (11) at (-1,0) [draw,circle] {};
\node (12) at (0,-1) [draw,circle] {};
\node (13) at (1,0) [draw,circle] {};
\node (14) at (2,1) [draw,circle] {};
\node (15) at (3,2) [draw] {};
\node (16) at (0,1) [draw,circle] {};
\node (17) at (0,3) [draw,circle] {};
\node (18) at (0,5) [draw,circle] {};
\begin{scope}[>=latex]
\draw [->] (1) to (2);
\draw [->] (2) to (3);

\draw [->] (4) to (5);
\draw [->] (5) to (6);
\draw [->] (6) to (7);
\draw [->] (7) to (8);

\draw [->] (9) to (10);
\draw [->] (10) to (11);
\draw [->] (11) to (12);
\draw [->] (12) to (13);
\draw [->] (13) to (14);
\draw [->] (14) to (15);

\draw [->] (15) to (18);
\draw [->] (18) to (9);

\draw [->] (8) to (14);
\draw [->] (14) to (17);
\draw [->] (17) to (10);
\draw [->] (10) to (4);

\draw [->] (3) to (7);
\draw [->] (7) to (13);
\draw [->] (13) to (16);
\draw [->] (16) to (11);
\draw [->] (11) to (5);
\draw [->] (5) to (1);

\draw [->] (6) to (11);
\draw [->] (11) to (17);
\draw [->] (17) to (13);
\draw [->] (13) to (6);

\draw [->] (2) to (5);
\draw [->] (5) to (10);
\draw [->] (10) to (18);
\draw [->] (18) to (14);
\draw [->] (14) to (7);
\draw [->] (7) to (2);

\draw [->, dashed] (1) to (4);
\draw [->, dashed] (4) to (9);
\draw [->, dashed] (8) to (3);
\draw [->, dashed] (15) to (8);
\end{scope}
\end{tikzpicture}
\caption{$\qD_3$-quiver.}
\label{fig-A3}
\end{figure}

We introduce a labelling of vertices of $\qD_n$, following \cite{SS19}. The frozen vertices on the left column are labelled by $V_{i,-i}$ with $1\leq i\leq n$, counting from top to bottom. Starting from a frozen vertex $V_{i,-i}$, we follow the arrows of the quiver in the South-East direction until we hit a vertex on the central column. We label the vertices along the path by $V_{i,r}$ with $r=-i,\cdots,0$. Then we start from the central vertex $V_{i,0}$ and follow the arrows in the North-East direction until we hit a frozen vertex on the right column. Label the vertices along the path by $V_{i,r}$ with $r=0,\cdots,i$. Next we rotate the quiver by $180^\circ$ and label the image of the vertex $V_{i,r}$ by $\Lambda_{i,r}$. In this way every vertices off the central column get labelled twice and those on the central column get labelled once. We write $V_i=V_{i,-i}$ and $\Lambda_{i}=\Lambda_{i,-i}$, $1\leq i\leq n$, to denote frozen vertices on the left column and right column, respectively.

Recall $\mathbb{F}=\mathbb{Q}(q^{1/2})$ and $\A=\mathbb{Z}[q^{1/2},q^{-1/2}]$. Let us write $_{\mathbb{F}}\mathcal{O}_q(\mathcal{X}_{|\qD_n|})=\mathbb{F}\otimes_{\A}\mathcal{O}_q(\mathcal{X}_{|\qD_n|})$ to be the base change of the quantum cluster algebra associated with the quiver $\qD_n$. Recall the notation $\ren{M}$ for the renormalized monomial in Section \ref{sec:defcl}. The following theorem is the cluster realisation of the quantum groups $\tU$.

\begin{theorem}\label{thm:tqg}(cf. \cites{SS19,Sh22}) 
    There is a $\Qq$-algebra embedding $t:\tU\hookrightarrow {_{\mathbb{F}}\mathcal{O}_q(\mathcal{X}_{|\qD_n|})}$, defined by 
    \begin{align*}
        &E_i\mapsto\sum_{r=-i}^{i-1}\ren{V_{i,-i}V_{i,-i+1}\cdots V_{i,r}},\\
        &K_i\mapsto \ren{V_{i,-i}V_{i,-i+1}\cdots V_{i,i}},\\
        &F_{n+1-i}\mapsto \sum_{r=-i}^{i-1}\ren{\Lambda_{i,-i}\Lambda_{i,-i+1}\cdots\Lambda_{i,r}},\\
        &K'_{n+1-i}\mapsto \ren{\Lambda_{i,-i}\Lambda_{i,-i+1}\cdots\Lambda_{i,i}},
    \end{align*}
    for $1\leq i\leq n$. Moreover under the embedding $t$, one has 
    \[
    \ZtU=\tU\cap \mathcal{O}_q(\mathcal{X}_{|\qD_n|}).
    \]
    
\end{theorem}

\section{Bases and integral forms}\label{sec:4}

In this section, we construct (rescaled) PBW-bases for $\imath$quantum groups, and study the integral forms spanned by such bases. We follow the notations in the previous section.

\subsection{PBW-bases}\label{sec:pbw}

Recall the braid group symmetries $\bT_i$ in \eqref{braid}. For a reduced expression $\mathbf{i}=(i_1,\cdots,i_r)$ of the longest element $w_0$ in the Weyl group $S_{n+1}$, define
\begin{equation}\label{eq:PBW}
    B_{\mathbf{i}}(\mathbf{a})=B_{i_1}^{a_1}\bT_{i_1}(B_{i_2}^{a_2})\cdots \bT_{i_1}\cdots \bT_{i_{r-1}}(B_{i_r}^{a_r}),\quad \text{and}\quad k(\mathbf{b})=k_1^{b_1}k_2^{b_2}\cdots k_n^{b_n},
\end{equation}
in $\dUi$, for $\mathbf{a}=(a_1,\dots,a_r)\in \mathbb{N}^r$ and $\mathbf{b}=(b_1,\dots,b_r)\in\mathbb{Z}^n$.

Recall the grading of the algebra $\tU$ in \eqref{eq:grad}.

\begin{lemma}\label{le:exp}
    We have
    \begin{equation}\label{eq:exp}
        B_{\mathbf{i}}(\mathbf{a})=F_{\mathbf{i}}(\mathbf{a})+\text{lower terms},\quad \text{for }\mathbf{a}\in\mathbb{N}^r.
    \end{equation}
    Here lower terms stand for $\Qq$-linear combinations of elements in $\tU$ with degrees strictly smaller than the degree of $F_{\mathbf{i}}(\mathbf{a})$. 
\end{lemma}

\begin{proof}
    For $1\leq k\leq r$, let $\mathbf{e}_k=(\delta_{ik})_{i=1}^r\in\mathbf{N}^r$. It  suffices to prove the lemma for $\mathbf{a}=\mathbf{e}_k$, for various $k$. We prove this by induction on $k$. Let $\mathbf{i}=(i_1,\cdots,i_r)$.

    The case when $k=1$ or $2$ follows directly from the definitions. Suppose $k>2$ and the lemma is proved for smaller $k$. 
    
    If $|i_{k-1}-i_k|>1$, then 
    \begin{equation*}
        B_{\mathbf{i}}(e_k)=\bT_{i_1}\cdots \bT_{i_{k-1}}(B_{i_k})=\bT_{i_1}\cdots \bT_{i_{k-2}}(B_{i_k})=B_{\mathbf{i'}}(\mathbf{e}_{k-1}),
    \end{equation*}
    where $\mathbf{i'}$ is obtained from $\mathbf{i}$ by swapping $(k-1)$-th and $k$-th terms. Similarly one has $F_{\mathbf{i}}(\mathbf{e}_k)=F_{\mathbf{i'}}(\mathbf{e}_{k-1})$. Therefore \eqref{eq:exp} holds by induction hypothesis.

    Now assume $|i_{k-1}-i_k|=1$. We separate the proof into two cases regarding whether the root $\beta=s_{i_1}\cdots s_{i_{k-2}}(\alpha_{i_k})$ is positive or negative.

    If $\beta$ is a positive root, we can take a reduced expression $\mathbf{i''}$ of $w_0$, such that the first $k-1$ factors are $i_1,\cdots ,i_{k-2},i_k$. Then we have
    \begin{equation*}
    \begin{split}
        B_{\mathbf{i}}(\mathbf{e}_k)=&\bT_{i_1}\cdots\bT_{i_{k-1}}(B_{i_k})\\
        =&\bT_{i_1}\cdots \bT_{i_{k-2}}\big(\frac{q^{1/2}B_{i_{k-1}}B_{i_k}-q^{-1/2}B_{i_k}B_{i_{k-1}}}{q-q^{-1}}\big)\\
        =&\frac{q^{1/2}B_{\mathbf{i}}(\mathbf{e}_{k-1})B_{\mathbf{i''}}(\mathbf{e}_{k-1})-q^{-1/2}B_{\mathbf{i''}}(\mathbf{e}_{k-1})B_{\mathbf{i}}(\mathbf{e}_{k-1})}{q-q^{-1}}\\
        \stackrel{(\heartsuit)}{=}&\frac{q^{1/2}F_{\mathbf{i}}(\mathbf{e}_{k-1})F_{\mathbf{i''}}(\mathbf{e}_{k-1})-q^{-1/2}F_{\mathbf{i''}}(\mathbf{e}_{k-1})F_{\mathbf{i}}(\mathbf{e}_{k-1})}{q-q^{-1}}
        +\text{lower terms}\\
        =&F_{\mathbf{i}}(\mathbf{e}_k)+\text{lower terms}.
    \end{split}
    \end{equation*}
    Here $(\heartsuit)$ follows from the induction hypothesis. 

    If $\beta$ is a negative root, then by possibly applying some braid group relations, we may assume $i_{k-2}=i_k$. By \cite{WZ22}*{Theorem 7.13} and \cite{Jan96}*{Lemma 8.20}, we have
    \begin{equation*}
        \bT_{i_k}\bT_{i_{k-1}}(B_{i_k})=B_{i_{k-1}}\quad \text{and}\quad \oT_{i_k}\oT_{i_{k-1}}(F_{i_k})=F_{i_{k-1}}.
    \end{equation*}
    Hence we have
    \begin{equation*}
    \begin{split}
        B_{\mathbf{i}}(\mathbf{e}_k)=&\bT_{i_1}\cdots\bT_{i_{k-3}}\bT_{i_k}\bT_{i_{k-1}}(B_{i_k})\\
        =&\bT_{i_1}\cdots\bT_{i_{k-3}}(B_{i_{k-1}})\\
        \overset{(\heartsuit')}{=}&\sigma_{i_1}\cdots\sigma_{i_{k-3}}(F_{i_{k-1}})+\text{lower terms}\\
        =&F_{\mathbf{i}}(\mathbf{e}_k)+\text{lower terms}.
    \end{split}
    \end{equation*}
    Here $(\heartsuit')$ follows from the induction hypothesis. 

    We complete the proof of the lemma.
\end{proof}

We next show that elements constructed in \eqref{eq:PBW} form a basis.

\begin{prop}\label{prop:PBWb}
For any reduced expression $\mathbf{i}$, the set
\begin{equation}\label{eq:PBWb}
    \{B_{\mathbf{i}}(\mathbf{a})k(\mathbf{b})\mid \mathbf{a}\in\mathbb{N}^r,\mathbf{b}\in\mathbb{Z}^n\}
\end{equation}
forms a $\Qq$-linear basis of the algebra $\tUi$.
\end{prop}

\begin{proof}
Thanks to Lemma \ref{le:exp} and the PBW-bases \eqref{eq:PBW0} of $\tU$, we deduce that elements in \eqref{eq:PBWb} are linearly independent. Recall the central subalgebra $\tUio$ in Section \ref{sec:iqg}. It is easy to see that $\{k(\mathbf{b})\mid \mathbf{b}\in\mathbb{Z}^n\}$ forms a linear basis of $\tUio$. Therefore we only need to show that $\tUi$ is spanned by elements $B_{\mathbf{i}}(\mathbf{a})$, for $\mathbf{a}\in\mathbb{N}^r$, as a $\tUio$-module. 

For a multi-index $J=(j_1,j_2,\cdots,j_k)\in [1,n]^k$, we denote $F_J=F_{j_1}F_{j_2}\cdots F_{j_k}$ and $B_J=B_{j_1}B_{j_2}\cdots B_{j_k}$. It is clear that
\begin{equation}\label{eq:bfi}
    B_J=F_J+\text{lower terms},
\end{equation}
where lower terms are linear combinations of elements of degrees smaller than the degree of $F_J$.

Take a subset $\mathcal{J}\subset \cup_{k=0}^\infty[1,n]^k$ such that the set $\{F_J\mid J\in\mathcal{J}\}$ forms a $\Qq$-basis of $\tU^{<}$. Then by \cite{WZ22}*{Proposition 2.6}, the set $\{B_J\mid J\in\mathcal{J}\}$ forms a $\tUio$-basis of $\tUi$. 

Define a filtration $\mathcal{F}^*$ on $\tUi$ by setting $\mathcal{F}^k(\tUi)$ to be the $\tUio$-span of elements $B_J$, with $J\in\mathcal{J}$ and $|J|\leq k$. Here we denote $|J|$ to be the length of the index $J$.

Take $\mathbf{a}=(a_1,\cdots,a_r)\in\mathbb{N}^r$. We expand
\begin{equation}\label{eq:exbk}
    B_{\mathbf{i}}(\mathbf{a})=\sum_{J\in\mathcal{J}}B_Jk_{\mathbf{a},J},
\end{equation}
with $k_{\mathbf{a},J}\in\tUio$. We expand both sides of \eqref{eq:exbk} in terms of triangular decomposition \eqref{eq:tria} of $\tU$, and compare the terms in the subalgebra $\tU^{<}$. By Lemma \ref{le:exp} and \eqref{eq:bfi}, we deduce that 
\begin{equation}\label{eq:comp}
    \alpha_{j_1}+\alpha_{j_2}+\cdots+\alpha_{j_k}\preceq a_1\beta_1+a_2\beta_2+\cdots+a_r\beta_r,
\end{equation}
for any $J=(j_1,\cdots,j_k)\in\mathcal{J}$ with $k_{\mathbf{a},J}\neq 0$. Here $\beta_t=s_{i_1}\cdots s_{i_{t-1}}(\alpha_t)$, for $1\leq t\leq r$. 

For any positive root $\gamma=\sum_{i=1}^nc_i\alpha_i$, let us denote $|\gamma|=\sum_{i=1}^nc_i$. Then \eqref{eq:comp} implies that 
\begin{equation*}
    |J|\leq a_1|\beta_1|+a_2|\beta_2|+\cdots+a_r|\beta_r|.
\end{equation*} 

Therefore we have 
\begin{equation*}
    B_{\mathbf{i}}(\mathbf{a})\in \mathcal{F}^{n(\mathbf{i},\mathbf{a})}(\tUi),\quad\text{where }n(\mathbf{i},\mathbf{a})=a_1|\beta_1|+a_2|\beta_2|+\cdots+a_r|\beta_r|.
\end{equation*}

On the other hand, by the definition of the filtration we have 
\begin{equation*}
    \text{dim}_{\tUio}(\mathcal{F}^k(\tUi))=|\{J\in\mathcal{J}\mid |J|\leq k\}|=|\{\mathbf{a}\in\mathbb{N}^r\mid n(\mathbf{i},\mathbf{a})\leq k\}|,
\end{equation*}
where the second equality follows from the PBW theorem of the algebra $\tU^<$.

Since elements $B_{\mathbf{i}}(\mathbf{a})$ ($\mathbf{a}\in\mathbb{N}^r$) are $\tUio$-linearly independent, we deduce that the set $\{B_{\mathbf{i}}(\mathbf{a})\mid n(\mathbf{i},\mathbf{a})\leq k\}$ spans $\mathcal{F}^k(\tUi)$ as a $\tUio$-module, by counting the dimension. This completes the proof.
\end{proof}

The basis in Proposition \ref{prop:PBWb} is called the \emph{PBW-basis} of $\tUi$ associated with the reduced expression $\mathbf{i}$. 

By the central reduction $\pi^\imath$ (see \eqref{eq:icen}), the set 
\begin{equation}
    \{B_{\mathbf{i}}(\mathbf{a})\mid \mathbf{a}\in\mathbb{N}^r\}
\end{equation}
forms a $\Qq$-linear basis of the algebra $\dUi$, which is referred to as the PBW-basis of $\dUi$ associated with the reduced expression $\mathbf{i}$.


\subsection{Quasi $K$-matrices and quantum dilogarithm}\label{sec:qK}

The \emph{quantum dilogarithm} $\Psi_q(x)$ plays important role in cluster algebras (see \cite{FG09}). It is defined to be the formal power series
\begin{equation}
    \Psi_q(x)=\frac{1}{(1+qx)(1+q^3x)(1+q^5x)\cdots}.
\end{equation}
It is the unique element in $\mathbb{Q}(q)[[x]]$ with the constant term 1 and the difference relation
\begin{equation}\label{eq:diff}
    \Psi(q^2x)=(1+qx)\Psi_q(x).
\end{equation}

By this characterisation, one can verify that $\Psi_q(x)$ has the expansion
\begin{equation}
    \Psi_q(x)=\sum_{n\geq 0}\frac{q^{-n(n-1)/2}}{(q-q^{-1})(q^2-q^{-2})\cdots(q^n-q^{-n})}x^n.
\end{equation}

Suppose we have $xy=q^{2\lambda}yx$, for some integer $\lambda$ in a non-commutative integral domain over $\mathbb{Q}(q)$. Then we have
\[  
\Psi_q(x)y=y\Psi_q(q^{2\lambda}x),
\]
in a certain completion of the algebra. Hence by \eqref{eq:diff} we get
\begin{equation}\label{eq:qadj}
    \text{Ad}_{\Psi_q(x)}(y)=\begin{cases}
        y(1+qx)(1+q^3x)\cdots(1+q^{2\lambda-1}x), &\text{if }\lambda> 0,\\
        y, &\text{if }\lambda=0,\\
        y\big((1+q^{-1}x)(1+q^{-3}x)\cdots(1+q^{2\lambda+1}x)\big)^{-1}, &\text{if }\lambda<0.
    \end{cases}
\end{equation}

Recall the quasi $K$-matrices in Section \ref{sec:K}. The following lemma relates quasi $K$-matrices with quantum dilogarithm.

\begin{lemma}\label{le:Klog}
For $1\leq i\leq n$ we have
\begin{equation}
    \tK_i=\Psi_{q^2}(-E_i^2).
\end{equation}
\end{lemma}

\begin{proof}
    Write $\tK_i=\sum_{n\geq 0}a_{2n}E_i^{2n}$, with $a_{2n}\in\Qq$ and $a_0=1$. By \eqref{eq:it} one has
    \begin{equation*}
        (F_i-q^{-1}E_iK_i')(\sum_{n\geq 0}a_{2n}E_i^{2n})=(\sum_{n\geq 0}a_{2n}E_i^{2n})(F_i-qE_iK_i).
    \end{equation*}

    By the degree consideration, we have
    \begin{equation}\label{eq:a2n}
        a_{2n}F_iE_i^{2n}-q^{-4n+3}a_{2n-2}E_i^{2n-1}K_i'=a_{2n}E_i^{2n}F_i-qa_{2n-2}E_i^{2n-1}K_i,   
    \end{equation}
    for $n\geq 1$.

    The following equality can be verified directly by induction on $n$,
    \begin{equation*}
        F_iE_i^{2n}=E_i^{2n}F_i+(q^{4n-1}-q^{-1})E_i^{2n-1}K_i+(q^{-4n+1}-q)E_i^{2n-1}K_i'.
    \end{equation*}

    Plug into \eqref{eq:a2n}. By the triangular decomposition, we get
    \begin{equation*}
        a_{2n}=\frac{-q^{-2n+2}}{q^{2n}-q^{-2n}}a_{2n-2}.
    \end{equation*}

    Hence we have
    \begin{equation*}
        \tK_i=\sum_{n\geq 0}\frac{q^{-n(n-1)}}{(q^2-q^{-2})(q^4-q^{-4})\cdots (q^{2n}-q^{-2n})}(-E_i^2)^n=\Psi_{q^2}(-E_i^2).
    \end{equation*}
\end{proof}

\begin{remark}\label{rmk:qkd}
    By \cite{WZ22}*{Theorem 8.1 (2)} (see also \cite{DK19}*{Theorem A}), the quasi $K$-matrix $\widetilde{\Upsilon}$ of type AI admits a factorisation into a product of rank one quasi $K$-matrices. In light of Lemma \ref{le:Klog}, we can then write $\widetilde{\Upsilon}$ as a product of quantum dilogarithm:
    \begin{equation}\label{eq:qkd}
    \widetilde{\Upsilon}=\overset{\rightarrow}{\prod_{\alpha\in \mathcal{R}_+}}\Psi_{q^2}(-E_\alpha^2).
    \end{equation}
    Here the product is taken in the total order associated with a reduced expression of the longest element, and $E_\alpha$ are (rescaled) root vectors associated with the reduced expression. This factorisation is analogous to the well-known factorisation of quasi $R$-matrix (cf. \cite{SS19}*{(36)}).
\end{remark}

\subsection{Integral forms}\label{sec:intf}

Recall $\A=\mathbb{Z}[q^{1/2},q^{-1/2}]$, and recall the rescaled integral form $\ZtU$ in Section \ref{sec:braid}. We consider the integral form of the $\imath$quantum group induced by $\ZtU$.

\begin{defi}\label{def:int}
The integral form $\ZtUi$ of $\tUi$ is defined to be the intersection
\begin{equation}
    \ZtUi=\tUi\cap\ZtU.
    \end{equation}
\end{defi}

Then $\ZtUi$ is an $\A$-subalgebra of $\tUi$. We next show that this subalgebra is invariant under braid group symmetries of $\imath$quantum groups. 

Recall the algebra automorphisms $\{\bT_i\mid 1\leq i\leq n\}$ in Section \ref{sec:braid}.

\begin{prop}\label{prop:invT}
    The $\A$-subalgebra $\ZtUi$ is invariant under the actions $\bT_i$, $\bT_i^{-1}$, for $1\leq i\leq n$.
\end{prop}

\begin{proof}
    We firstly show that $\ZtUi$ is invariant under actions $\bT_i^{-1}$, for $1\leq i\leq n$.

    Take any $x\in\ZtUi$. It will suffice to show that $\bT_i^{-1}(x)$ belongs to $\ZtU$. Recall the embedding $t$ in Theorem \ref{thm:tqg}. By \cite{SS19}*{Remark 3}, after applying a sequence of mutations, we get the quiver $\qD_n^i$, such that $t(E_i)$ becomes the frozen variable $V_i$ in $\T_{\qD_n^i}$. Suppose $t(\bT_i^{-1}(x))$ has the expansion
\begin{equation}\label{eq:ext}
    t(\bT_i^{-1}(x))=\sum_v c_v X_v\qquad \text{in }\T_{\qD_n^i}.
\end{equation}
Here $X_v$'s are Laurent monomials in $\T_{\qD_n^i}$ and $c_v\in\Qq$. It will suffice to show that each $c_v$ belongs to $\A$, thanks to Theorem \ref{thm:tqg}. 



For any quiver $Q$, which is mutation-equivalent to $\qD_n$, let $\tT'_Q$ be the sub skew-field of the fraction field $\tT_Q$, which is generated by all unfrozen variables, frozen variables $\Lambda_k$ ($1\leq k\leq n$) on the right column, and the squares of the frozen variables $V_k^2$ ($1\leq k\leq n$) on the left column. It is clear that $\tT'_Q$ is preserved under mutations, that is, $\mu_k(\tT'_Q)=\tT_{\mu_kQ}'$ for any unfrozen vertex $k$. Therefore $\tT_Q'$ is preserved under any quasi-cluster automorphisms. 

It is direct to verify that under the map $t$, images of elements $B_i=F_i-q^{-1}E_iK_i'$, $k_i=K_iK_i'$, for $1\leq i\leq n$, belong to $\tT'_{\qD_n}$. Therefore in the cluster chart associated with $\qD_n^i$, the image $t(\tUi)$ is contained in $\tT_{\qD_n^i}'$. In particular $t(x)\in\tT_{\qD_n^i}'$. By \cite{GS22}*{13.6}, Lusztig's braid group action $\sigma_i^{-1}$ is a quasi-cluster automorphism. Therefore we deduce $t(\sigma^{-1}_i(x))\in\tT_{\qD_n^i}'$. Note that $\sigma_i^{-1}(x)\in\ZtU$. Therefore when expanding in $\T_{\qD_n^i}$, the element $t(\sigma_i^{-1}(x))$ is an $\A$-linear combination of Laurent monomials, where the degrees of $V_k$, for $1\leq k\leq n$, are even, in each monomials shown in the expansion. 

Next we consider the adjoint action $\text{Ad}_{\Psi_{q^2}(-V_i^2)}$ of $\tT_{\qD_n^i}'$. It is easy to see that for any generator $Y$ (which is either the square of frozen variables on the left column, or frozen variables on the right column, or unfrozen variables) of $\tT_{\qD_n^i}'$, we have commuting relation $V_i^2Y=q^{4\lambda}YV_i^2$ for some integer $\lambda$. Therefore by \eqref{eq:qadj}, we have the following equality in $\tT_{\qD_n^i}$,
\begin{equation}\label{eq:adp}
    \text{Ad}_{\Psi_{q^2}(-V_i^2)}(t(\sigma_i^{-1}(x)))=\big(\sum_{v'}b_{v'}X_{v'}\big)\big((1-q^2V_i^2)(1+q^6V_i^4)\cdots(1+(-1)^sq^{4s-2}V_i^{2s})\big)^{-1}.
\end{equation}
Here $s$ is some positive integer, $X_{{v'}}$'s are Laurent polynomials in $\T_{\qD_n^i}$ and coefficients $b_{v'}$ belong to $\A$.

Recall that $t(E_i)=V_i$ in $\T_{\qD_n^i}$. By \eqref{eq:intw} and Lemma \ref{le:Klog},  
we have
\begin{equation}
    t(\bT_i^{-1}(x))=t(\text{Ad}_{\Psi_{q^2}(-E_i^2)}\circ\sigma_i^{-1}(x))=\text{Ad}_{\Psi_{q^2}(-V_i^2)}(t(\sigma_i^{-1}(x)))\quad \text{in }\tT_{\qD_n^i}.
\end{equation}

Compare the right hand sides of \eqref{eq:ext} and \eqref{eq:adp}. We get
\begin{equation}
    \big(\sum_vc_vX_v\big)(1-q^2V_i^2)(1+q^6V_i^4)\cdots(1+(-1)^sq^{4s-2}V_i^{2s})=\sum_{v'}b_{v'}X_{v'}.
\end{equation}
Since $b_{v'}\in \A$, it is then easy to deduce that each $c_v$ belongs to $\A$. We proved that $\ZtUi$ is invariant under $\bT_i^{-1}$, for $1\leq i\leq n$.

By the identity \eqref{eq:cyc} and an induction on $i$, we can show that $\ZtUi$ is also invariant under $\bT_i$, for $1\leq i\leq n$. 
\end{proof}

We next show that the integral form $\ZtUi$ is the $\A$-span of the bases in Proposition \ref{prop:PBWb}.

\begin{cor}\label{cor:Abase}
    For any reduced expression $\mathbf{i}$ of the longest element $w_0$, the set
\begin{equation}\label{eq:PBWb1}
    \{B_{\mathbf{i}}(\mathbf{a})k(\mathbf{b})\mid \mathbf{a}\in\mathbb{N}^r,\mathbf{b}\in\mathbb{Z}^n\}
\end{equation}
gives an $\A$-linear basis of the integral form $\ZtUi$.
\end{cor}

\begin{proof}
    It is obvious that generators $B_i$, $k_i$, for $1\leq i\leq n$, belong to $\ZtUi$. Then by Proposition \ref{prop:invT}, elements in \eqref{eq:PBWb1} belong to $\ZtUi$. Thanks to Proposition \ref{prop:PBWb}, elements in \eqref{eq:PBWb1} are linearly independent. It will suffice to show that any element in $\ZtUi$ can be written as an $\A$-linear combination of elements in \eqref{eq:PBWb1}. 
    
    Take any element $x$ in $\ZtUi$. We expand $x=\sum_bc_bb$, where the summation is over elements in \eqref{eq:PBWb1}, and $c_b\in\Qq$. Since $x$ belongs to $\ZtU$, its expansion in terms of PBW-bases of $\tU$ (see \eqref{eq:PBW0}) has coefficients in $\A$. Together with Lemma \ref{le:exp}, it is easy to deduce that each $c_b$ also belongs to $\A$. This completes the proof.
\end{proof}

The following lemma gives a characterization for the integral form $\ZtUi$, which will be needed in Corollary \ref{cor:emq}.

\begin{lemma}\label{le:inti}
The integral form $\ZtUi$ is the smallest $\A$-subalgebra of $\tUi$, which satisfies the following conditions:

a) It contains generators $B_i$, $k_i$, for $1\leq i\leq n$;

b) It is invariant under automorphisms $\bT_i$, for $1\leq i\leq n$.
\end{lemma}

\begin{proof}
    It follows from Proposition \ref{prop:invT} and Corollary \ref{cor:Abase}.
\end{proof}

We next explain that integral forms and PBW-bases of the universal $\imath$quantum groups descend to integral forms and PBW-bases of the $\imath$quantum group, respectively.

Recall the $\Qq$-algebra $\dUi$ and the central reduction $\pi^\imath$ in Section \ref{sec:iqg}. Let us define 
\begin{equation}\label{eq:intd}
    \ZdUi=\pi^\imath(\ZtUi)
\end{equation}
to be the integral form of $\dUi$. Recall the integral form $\ZdU$ of the algebra $\dU$ in Section \ref{sec:braid}.

Next proposition asserts that $\ZdUi$ is induced from the (rescaled) integral form of Drinfeld--Jimbo quantum group. 

\begin{prop}\label{prop:inti}
    We have
    \begin{equation*}
        \ZdUi=\dUi\cap\ZdU.
    \end{equation*}
\end{prop}

\begin{proof}
    Set $\mathbb{F}'=\Qq(\sqrt{-1})$ to be the field extension of $\Qq$ by joining the square root of $-1$, and set
    \[
    \A[\sqrt{-1}]=\{a+b\sqrt{-1}\mid a,b\in\A\}\subset \mathbb{F}'.
    \]

    We denote by $_{\mathbb{F}'}\tU$, $_{\mathbb{F}'}\dU$, $_{\mathbb{F}'}\tUi$, $_{\mathbb{F}'}\dUi$ to be the $\mathbb{F}'$-algebras, obtained by the field extensions of the corresponding $\Qq$-algebras. Similarly set $_{\A[\sqrt{-1}]}\tU$, $_{\A[\sqrt{-1}]}\dU$, $_{\A[\sqrt{-1}]}\tUi$, $_{\A[\sqrt{-1}]}\dUi$ to be the $\A[\sqrt{-1}]$-algebras, obtained by base change of the corresponding $\A$-algebras.
    
    Let $\Phi$ be the $\mathbb{F}'$-algebra automorphism of $_{\mathbb{F}'}\tU$, given by
    \[
    F_i\mapsto F_i,\quad E_i\mapsto\sqrt{-1}E_i,\quad K_i\mapsto \sqrt{-1}K_i,\quad K_i'\mapsto \sqrt{-1}K_i',\quad \text{for }1\leq i\leq n.
    \]

    Then the central reduction $\pi^\imath$ coincides with the restriction of $\pi\circ\Phi$ to the subalgebra $_{\mathbb{F}'}\tUi$ (cf. \cite{WZ22}*{Remark 2.8}). Here by abuse of notation, $\pi$ and $\pi^\imath$ denote the induced maps after the field extension. Since $\Phi$ preserves the integral form $_{\A[\sqrt{-1}]}\tU$, it follows that
    \begin{equation}\label{eq:fi}
        _{\A[\sqrt{-1}]}\dUi={_\mathbb{F}\dUi}\cap{_{\A[\sqrt{-1}]}\dU}.
    \end{equation}

    Let us decompose both sides of \eqref{eq:fi} as $\A$-modules
    \[
    _{\A[\sqrt{-1}]}\dUi={\ZdUi}\oplus \sqrt{-1}\ZdUi,
    \]
    and
    \[
    {_\mathbb{F}\dUi}\cap{_{\A[\sqrt{-1}]}\dU}=(\dUi\cap\ZdU)\oplus \sqrt{-1}(\dUi\cap\ZdU).
    \]
    Then the proposition follows easily.
\end{proof}

Recall that braid group symmetries $\{\bT_i\mid 1\leq i\leq n\}$ descend to the algebra $\dUi$.

\begin{cor}\label{cor:iau}
    The integral form $\ZdUi$ is invariant under braid group symmetries $T_i$, for $1\leq i\leq n$. Moreover, for any reduced expression $\mathbf{i}$ of $w_0$, the set
    \begin{equation*}
    \{B_{\mathbf{i}}(\mathbf{a})\mid \mathbf{a}\in\mathbb{N}^r\}
\end{equation*}   
    forms an $\A$-linear basis of $\ZdUi$.
\end{cor}

\section{Cluster realisations of $\imath$quantum groups}\label{sec:5}

In this section, we give cluster realisations of $\imath$quantum groups, and give cluster interpretations of fundamental constructions. Retain the same notations as in the previous sections.

\subsection{The quiver $\Sigma_n$}\label{sec:quiv}

Let us firstly introduce the related quivers of our cluster realisations. Recall the quiver $\widetilde{\Sigma_n}$ associated with the $n$-triangulation of a triangle $ABC$, in Section \ref{sec:clq}. We label the vertices of $\widetilde{\Sigma_n}$ on the edge $AB$ by $1',2',\dots,n'$ in the direction $\overrightarrow{BA}$, and label the edge $AC$ by $1'',2'',\dots,n''$ in the direction $\overrightarrow{AC}$. The quiver $\Sigma_n$ is obtained by the self-amalgamation (see Section \ref{sec:defcl}) of $\widetilde{\Sigma_n}$, by identifying frozen vertices $i'$ with $i''$, for $1\leq i\leq n$. 

\begin{figure}[h]
    \centering
    \begin{tikzpicture}[every node/.style={inner sep=0, minimum size=0.4cm, thick, draw, fill=white}, thick, x=0.9cm, y=1.2cm]
        \node (1) at (3,0.5) [draw] {$\tiny{2}$};
        \node (2) at (1,0.5) [draw] {$\tiny{1'}$};
        \node (3) at (2,-0.5) [draw] {$\tiny{1''}$};

        \begin{scope}[>=latex]
            \qarrow{1}{2};
            \qarrow{2}{3};
            \qarrow{3}{1};
        \end{scope}

        \draw [shade, top color = gray!35, shading angle=-90, thick] (4,-0.2) to (4,0.2) to (5.7,0.2) to (5.7,0.3) to (6.2,0) to (5.7,-0.3) to (5.7,-0.2) to (4,-0.2);
        \node [draw=none, fill=none] at (5,0) {$\tiny{1'\leftrightarrow1''}$};

        \node (1) at (7.3,0) [draw, circle] {$\tiny{1}$};
        \node (2) at (9,0) [draw] {$\tiny{2}$};
    \end{tikzpicture}
    \caption{$\widetilde{\Sigma_1}$-quiver (left) and $\Sigma_1$-quiver (right)}
    \label{fig:s1}
\end{figure}
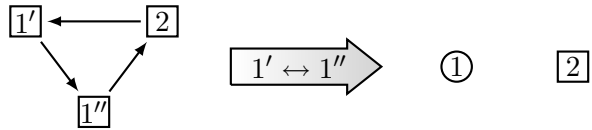

\begin{figure}[h]
\centering
\begin{tikzpicture}[every node/.style={inner sep=0, minimum size=0.4cm, thick, draw, fill=white}, thick, x=0.9cm, y=1.2cm]

\node (1) at (1.5,-1) [draw] {$\tiny{1''}$};
\node (2) at (3.5,-1) [draw] {$\tiny{2''}$};
\node (3) at (0.5,0) [draw] {$\tiny{2'}$};
\node (4) at (2.5,0) [draw,circle] {$\tiny{3}$};
\node (5) at (4.5,0) [draw] {$\tiny{4}$};
\node (6) at (1.5,1) [draw] {$\tiny{1'}$};
\node (7) at (3.5,1) [draw] {$\tiny{5}$};

\begin{scope}[>=latex]
\qarrow{7}{6};
\qdarrow{7}{5};
\qarrow{4}{7};
\qarrow{6}{4};
\qarrow{4}{3};
\qdarrow{3}{6};
\qarrow{5}{4};
\qarrow{3}{1};
\qarrow{1}{4};
\qarrow{4}{2};
\qarrow{2}{5};
\qdarrow{2}{1};
\end{scope}
    \draw [shade, top color = gray!35, shading angle=-90, thick] (5.5,-0.4) to (5.5,0.4) to (7.3,0.4) to (7.3,0.5) to (7.8,0) to (7.3,-0.5) to (7.3,-0.4) to (5.5,-0.4);

\node [draw=none, fill=none] at (6.5,0.2) {$\tiny{1'\leftrightarrow 1''}$};
\node [draw=none, fill=none] at (6.5,-0.2) {$\tiny{2'\leftrightarrow 2''}$};

    \node (1) at (12,0.8) [draw] {$\tiny{5}$};
    \node (2) at (9,1) [draw, circle] {$\tiny{1}$};
    \node (3) at (10.5,0) [draw, circle] {$\tiny{3}$};
    \node (4) at (12,-0.8) [draw] {$\tiny{4}$};
    \node (5) at (9,-1) [draw, circle] {$\tiny{2}$};

    \begin{scope}[>=latex]
        \qdarrow{1}{4};
        \qarrow{1}{2};
        \qarrow{3}{1};
        \qarrow{4}{3};
        \qarrow{5}{4};
        \qddarrow{2}{3};
        \qddarrow{3}{5};
        \qddarrow{5}{2};
    \end{scope}
\end{tikzpicture}
\caption{$\widetilde{\Sigma_2}$-quiver (left) and $\Sigma_2$-quiver (right)}
\label{fig:s2}
\end{figure}
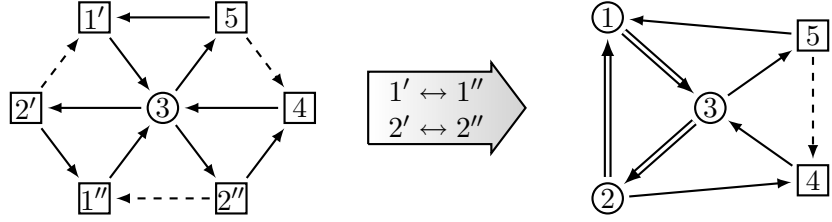

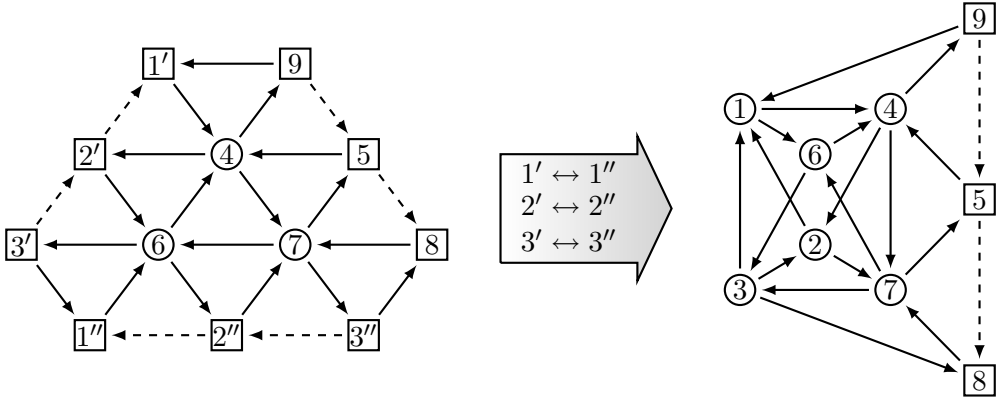
\begin{figure}[h] 
    \centering
    \begin{tikzpicture}[every node/.style={inner sep=0, minimum size=0.4cm, thick, draw, fill=white}, thick, x=0.9cm, y=1.2cm]
    \node (1) at (5,3) [draw] {$\tiny{9}$};
    \node (2) at (3,3) [draw] {$\tiny{1'}$};
    \node (5) at (6,2) [draw] {$\tiny{5}$};
    \node (4) at (4,2) [draw, circle] {$\tiny{4}$};
    \node (6) at (2,2) [draw] {$\tiny{2'}$};
    \node (8) at (7,1) [draw] {$\tiny{8}$};
    \node (7) at (5,1) [draw, circle] {$\tiny{7}$};
    \node (3) at (3,1) [draw, circle] {$\tiny{6}$};
    \node (9) at (1,1) [draw] {$\tiny{3'}$};
    \node (10) at (2,0) [draw] {$\tiny{1''}$};
    \node (11) at (4,0) [draw] {$\tiny{2''}$};
    \node (12) at (6,0) [draw] {$\tiny{3''}$};

    \begin{scope}[>=latex]
        \qdarrow{1}{5};
        \qdarrow{5}{8};
        \qdarrow{12}{11};
        \qdarrow{11}{10};
        \qdarrow{9}{6};
        \qdarrow{6}{2};
        \qarrow{1}{2};
        \qarrow{5}{4};
        \qarrow{4}{6};
        \qarrow{8}{7};
        \qarrow{7}{3};
        \qarrow{3}{9};
        \qarrow{2}{4};
        \qarrow{4}{7};
        \qarrow{7}{12};
        \qarrow{6}{3};
        \qarrow{3}{11};
        \qarrow{9}{10};
        \qarrow{12}{8};
        \qarrow{11}{7};
        \qarrow{7}{5};
        \qarrow{10}{3};
        \qarrow{3}{4};
        \qarrow{4}{1};
    \end{scope}

    \draw [shade, top color = gray!35, shading angle=-90, thick] (8,0.8) to (8,2) to (10,2) to (10,2.2) to (10.5,1.4) to (10,0.6) to (10,0.8) to (8,0.8);
    \node [draw=none, fill=none] at (9,1.8) {$1'\leftrightarrow1''$};
    \node [draw=none, fill=none] at (9,1.45) {$2'\leftrightarrow 2''$};
    \node [draw=none, fill=none] at (9,1.05) {$3'\leftrightarrow 3''$};

    \node (9) at (11.5, 0.5) [draw, circle] {$\tiny{3}$};
    \node (2) at (11.5, 2.5) [draw, circle] {$\tiny{1}$};
    \node (4) at (13.7, 2.5) [draw, circle] {$\tiny{4}$};
    \node (7) at (13.7, 0.5) [draw, circle] {$\tiny{7}$};
    \node (6) at (12.6, 1) [draw, circle] {$\tiny{2}$};
    \node (3) at (12.6, 2) [draw, circle] {$\tiny{6}$};
    \node (1) at (15, 3.5) [draw] {$\tiny{9}$};
    \node (5) at (15, 1.5) [draw] {$\tiny{5}$};
    \node (8) at (15, -0.5) [draw] {$\tiny{8}$};

    \begin{scope}[>=latex]
        \qdarrow{1}{5};
        \qdarrow{5}{8};
        \qarrow{1}{2};
        \qarrow{2}{3};
        \qarrow{3}{4};
        \qarrow{4}{1};
        \qarrow{5}{4};
        \qarrow{4}{6};
        \qarrow{6}{7};
        \qarrow{7}{5};
        \qarrow{8}{7};
        \qarrow{7}{3};
        \qarrow{3}{9};
        \qarrow{9}{8};
        \qarrow{2}{4};
        \qarrow{4}{7};
        \qarrow{7}{9};
        \qarrow{9}{2};
        \qarrow{6}{2};
        \qarrow{9}{6};
    \end{scope}
    
    \end{tikzpicture}
    \caption{$\widetilde{\Sigma_3}$-quiver (left) and $\Sigma_3$-quiver (right)}
    \label{fig:s3}
\end{figure}

Let us introduce a labelling of vertices of the quiver $\Sigma_n$. Note that all the frozen vertices are on the edge $BC$. We denote these vertices by $X_{i,0}$ for $i=1,\dots,n$, counting from $B$ to $C$. Choose a frozen vertex $X_{i,0}$ and follow the arrows in the direction $\overrightarrow{CA}$. Then we will hit a vertex on the edge $AB$ which is identified with a vertex on $AC$ by the self-amalgamation. We continue to follow the arrows in the direction $\overrightarrow{AB}$ and will get back to the frozen vertex $X_{i,0}$. In this way we get a loop with the base point $X_{i,0}$. We label the vertices along the loop by $X_{i,t}$, $t=0,1,\dots,n$. It is easy to verify the following relation
\begin{equation}
    X_{j,j-i}=X_{i,n+i-j+1},\qquad 1\leq i<j\leq n.\label{eq:relabel}
\end{equation}
By abuse of notation, we also use $X_{i,j}$ to denote corresponding generators of the associated quantum torus algebra $\T_{\Sigma_n}$.

\begin{example}
    In the Figures \ref{fig:s1}, \ref{fig:s2} and \ref{fig:s3}, let us refer to the $i$-th vertex by $X_i$. Then in the $\Sigma_1$-quiver (see Figure \ref{fig:s1}), our labelling is
    \begin{equation*}
        X_{1,0}=X_2,\quad X_{1,1}=X_1.
    \end{equation*}
    In the $\Sigma_2$-quiver (see Figure \ref{fig:s2}), our labelling is 
    \begin{equation*}
        \begin{array}{ccc}
          X_{1,0}=X_5,\quad & X_{1,1}=X_1, \quad & X_{1,2}=X_3,\\
        X_{2,0}=X_4, \quad & X_{2,1}=X_3, \quad & X_{2,2}=X_2.
        \end{array}
    \end{equation*}
    In the $\Sigma_3$-quiver (see Figure \ref{fig:s3}), our labelling is
    \begin{equation*}
        \begin{array}{cccc}
        X_{1,0}=X_9, \quad & X_{1,1}=X_1, \quad & X_{1,2}=X_6, \quad & X_{1,3}=X_4,\\
        X_{2,0}=X_5, \quad & X_{2,1}=X_4, \quad & X_{2,2}=X_2, \quad & X_{2,3}=X_7,\\
        X_{3,0}=X_8, \quad & X_{3,1}=X_7, \quad & X_{3,2}=X_6, \quad & X_{3,3}=X_3.
    \end{array}
    \end{equation*}
\end{example}

\subsection{An algebra embedding}\label{sec:albe}

Recall in Section \ref{sec:defcl}, that the quantum cluster algebra $\mathcal{O}_q(\mathcal{X}_{|\Sigma_n|})$ is an $\A$-subalgebra of the quantum torus algebra $\T_{\Sigma_n}$. Set $$\FOX=\Qq\otimes_\A\mathcal{O}_q(\mathcal{X}_{|\Sigma_n|})$$ to be the base change of the quantum cluster algebra. Recall the notation $\ren{M}$ for the renormalized monomial of $M$ in Section \ref{sec:defcl}.

The following is the main theorem in this section.

\begin{theorem}\label{thm:emb}
    There is a $\Qq$-algebra embedding $$\iota:\tUi\hookrightarrow\FOX,$$ given by
    \begin{align}
        &B_i\mapsto \sum_{k=0}^n:X_{i,0}X_{i,1}\cdots X_{i,k}:,\label{eq:em1}\\
        &k_i\mapsto -\ren{X_{i,0}^2X_{i,1}\cdots X_{i,n}},\label{eq:em2}
    \end{align} 
     for $1\leq i\leq n$.
\end{theorem}

\begin{remark}
    One can precisely write down the renormalized monomials as follows. For $n>1$ and $0\leq k\leq n$, we have
    $$:X_{i,0}X_{i,1}\cdots X_{i,k}:=
        q^kX_{i,0}X_{i,1}\cdots X_{i,k}, \quad 
    :X_{i,0}^2X_{i,1}\cdots X_{i,n}:=q^nX_{i,0}^2X_{i,1}\cdots X_{i,n}.$$
    For $n=1$, the quiver $\Sigma_{1}$ is disconnected, so the renormalization is not needed. 
\end{remark}

\begin{example}
    For $n=1$, in the notation of Figure \ref{fig:s1}, the embedding $\iota$ reads
    \begin{equation*}
        B_1\mapsto X_2(1+X_1),\quad\quad k_1\mapsto -X_2^2X_1.
        \end{equation*}

    For $n=2$, in the notation of Figure \ref{fig:s2}, the embedding $\iota$ reads
    \begin{align*}
        &B_1\mapsto X_5(1+qX_1(1+qX_3)), & k_1\mapsto -q^2X_5^2X_1X_3,\\
        & B_2\mapsto X_4(1+qX_3(1+qX_2)), & k_2\mapsto -q^2X_4^2X_3X_2.
    \end{align*}

    For $n=3$, in the notation of Figure \ref{fig:s3}, the embedding $\iota$ reads
    \begin{align*}
        B_1\mapsto X_9(1+qX_1(1+qX_6(1+qX_4))), \quad & k_1\mapsto -q^3X_9^2X_1X_6X_4, \\
        B_2 \mapsto X_5(1+qX_4(1+qX_2(1+qX_7))), \quad & k_2 \mapsto -q^3X_5^2X_4X_2X_7, \\
        B_3 \mapsto X_8(1+qX_7(1+qX_6(1+qX_3))), \quad & k_3 \mapsto -q^3X_8^2X_7X_6X_3.
    \end{align*}
\end{example}

The proof of Theorem \ref{thm:emb} will follow from Propositions \ref{prop:emq}, \ref{prop:inj} and \ref{prop:qcl} below.

In what follows, we set 
\begin{equation}\label{eq:cen}
    C_i=\ren{X_{i,0}^2X_{i,1}\cdots X_{i,n}}=\begin{cases}
        X_{i,0}^2X_{i,1} &\text{if }n=1,\\
        q^nX_{1,0}^2X_{i,1}\cdots X_{i,n} &\text{if }n\geq 2,
    \end{cases}
\end{equation}
for $1\leq i\leq n$. It is direct to verify that $C_i$ are central elements in $\T_{\Sigma_n}$.


The following proposition asserts that the defining relations of $\tUi$ are satisfied in $\T_{\Sigma_n}$, under the assignments \eqref{eq:em1} and \eqref{eq:em2}.

\begin{prop}\label{prop:emq}
The assignment \eqref{eq:em1} and \eqref{eq:em2} define a $\Qq$-algebra homomorphism $\iota:\tUi\rightarrow\T_{\Sigma_n}$.
\end{prop}

\begin{proof}
We need to verify relations \eqref{eq:R0} to \eqref{eq:R2} in $\T_{\Sigma_n}$. Since $C_i$ ($1\leq i\leq n$) are central elements in $\tT_{\Sigma_n}$, the relation \eqref{eq:R0} follows directly. We next verify the relations \eqref{eq:R1} and \eqref{eq:R2}.

For $1\leq i,k\leq n$, write 
\begin{equation}\label{eq:pik}
    P_{i,k}=:X_{i,0}X_{i,1}\cdots X_{i,k}:.
\end{equation}

Firstly take $1\leq i< j\leq n$ with $|i-j|>1$. The following commuting relations in $\T_{\Sigma_n}$ can be verified directly,
\begin{equation*}
    P_{i,n+i-j}P_{j,j-i}=q^{-2}P_{j,j-i}P_{i,n+i-j},\quad P_{i,n+i-j+1}P_{j,j-i-1}=q^2P_{j,j-i-1}P_{i,n+i-j+1},
\end{equation*}
and
\begin{equation*}
    [P_{i,t},P_{j,l}]=0,\quad\text{if }(t,l)\not\in\{(n+i-j,j-i),(n+i-j+1,j-i-1)\}.
\end{equation*}
Hence we have
\begin{align*}
    \iota([B_i,B_j])&=[\sum_{k=0}^nP_{i,k},\sum_{k=0}^nP_{j,k}]\\
    &=[P_{i,n+i-j},P_{j,j-i}]+[P_{i,n+i-j+1},P_{j,j-i-1}]\\
    &=(1-q^2)P_{i,n+i-j}P_{j,j-i}+(1-q^{-2})P_{i,n+i-j+1}P_{j,j-i-1}.
\end{align*}
Recall \eqref{eq:relabel}. We have
\begin{align*}
    P_{i,n+i-j}P_{j,j-i}=&q^{-1}\ren{X_{i,0}\cdots X_{i,n+i-j}X_{j,0}\cdots X_{j,j-i}},\\
    =& q^{-1}\ren{X_{i,0}\cdots X_{i,n+i-j+1}X_{j,0}\cdots X_{j,j-i-1}}\\
    =&q^{-2}P_{i,n+i-j+1}P_{j,j-i-1}.
\end{align*}
Recall the relation \eqref{eq:relabel}. We have
\begin{equation*}
    \iota([B_i,B_j])=((1-q^2)q^{-2}+(1-q^{-2}))P_{i,n+i-j+1,j,j-i-1}=0.
\end{equation*}
Hence we have checked the relation \eqref{eq:R1}. 

Let us now check the relation \eqref{eq:R2}. The following commuting relations can be verified directly.

If $(t,l)\neq (n-1,1)$ or $ (n,0)$, we have
\begin{align}\label{eq:com1}
    P_{i,t}P_{i+1,l}=q^{-1}P_{i+1,l}P_{i,t},\quad \text{for }l\leq t, \quad\text{and}\quad P_{i,t}P_{i+1,l}=qP_{i+1,l}P_{i,t},\quad \text{for }l>t.
\end{align}

We also have
\begin{equation}\label{eq:com2}
    P_{i,n-1}P_{i+1,1}=q^{-3}P_{i+1,1}P_{i,n-1},\qquad P_{i,n}P_{i+1,0}=qP_{i+1,0}P_{i,n},
\end{equation}
and
\begin{equation}\label{eq:com3}
    P_{i,l}P_{i,m}=q^{-2}P_{i,m}P_{i,l}\quad \text{if $l<m$ and $(l,m)\neq (0,n)$},\qquad [P_{i,0},P_{i,n}]=0.
\end{equation}

Fix $1\leq i <n$. For $1\leq t,l,m\leq n$, set
$$S_{t,l,m}=P_{i,t}P_{i+1,l}P_{i+1,m}-(q+q^{-1})P_{i+1,l}P_{i,t}P_{i+1,m}+P_{i+1,l}P_{i+1,m}P_{i,t}.$$
Then 
\begin{equation}\label{eq:Ser}
    \iota\big(B_iB_{i+1}^2-(q+q^{-1})B_{i+1}B_iB_{i+1}+B_{i+1}^2B_i\big)=\sum_{0\leq t,l,m\leq n}S_{t,l,m}.
\end{equation}

Firstly suppose $l=m$. By relations \eqref{eq:com1} and \eqref{eq:com2} we get
\begin{equation}\label{eq:ch1}
    S_{t,l,l}=\begin{cases}
        (1-q^2)(1-q^4)P_{i,n-1}P_{i+1,1}^2, &\text{if }(t,l)=(n-1,1),\\
        0, &\text{if } (t,l)\neq (n-1,1).
    \end{cases}
\end{equation}

Then suppose $l\neq m$. Assume further $(t,l),(t,m)\notin\{ (n-1,1),(n,0)\}$. By relations \eqref{eq:com1} and \eqref{eq:com2} we have
\begin{equation*}
    S_{t,l,m}=\begin{cases}
        (1-q^{2})P_{i,t}P_{i+1,l}P_{i+1,m}, &\text{if }m<t\leq l,\\
        (1-q^{-2})P_{i,t}P_{i+1,l}P_{i+1,m}, &\text{if }l<t\leq m,\\
        0, & \text{if otherwise.}
    \end{cases}
\end{equation*}

Hence under the assumption $l\neq m$ and $(t,l),(t,m)\notin\{ (n-1,1),(n,0)\}$, by \eqref{eq:com3} we have
\begin{equation}\label{eq:ch2}
    S_{t,l,m}+S_{t,m,l}=0,\quad \text{if }(l,m)\neq (0,n), (n,0),
\end{equation}
and
\begin{equation}\label{eq:ch3}
    S_{t,0,n}+S_{t,n,0}=-(q-q^{-1})^2P_{i,t}P_{i+1,0}P_{i+1,n}.
\end{equation}

By relations \eqref{eq:com1} and \eqref{eq:com2} we have
\begin{equation*}
    S_{n-1,1,m}=\begin{cases}
        (1-q^2)P_{i,n-1}P_{i+1,1}P_{i+1,m}, &\text{if }m\leq n-1\\
        (1-q^4)P_{i,n-1}P_{i+1,1}P_{i+1,n}, & \text{if $m=n$.}
    \end{cases}
\end{equation*}
We also have
\begin{equation*}
    S_{n-1,m,1}=\begin{cases}
        (-q^2+q^4)P_{i,n-1}P_{i+1,m}P_{i+1,1}, &\text{if }m\leq n-1\\
        (-q^{-2}+q^2)P_{i,n-1}P_{i+1,n}P_{i+1,1}, & \text{if }m=n.
    \end{cases}
\end{equation*}
Together with \eqref{eq:com3} we have 
\begin{equation}\label{eq:ch4}
    S_{n-1,1,m}+S_{n-1,m,1}=\begin{cases}
        0, &\text{if $m=0$ or $n$,}\\
        (1-q^2)(1-q^4)P_{i,n-1}P_{i+1,1}P_{i+1,m} &\text{if }1<m\leq n-1.
    \end{cases}
\end{equation}
Similarly we get
\begin{equation}\label{eq:ch5}
    S_{n,0,m}+S_{n,m,0}=\begin{cases}
        -(q-q^{-1})^2P_{i,n}P_{i+1,0}P_{i+1,n},& \text{if }m=n\\
        (1-q^{-2})(1-q^4)P_{i,n}P_{i+1,0}P_{i+1,m}, & \text{if }0<m\leq n.
    \end{cases}
\end{equation}

In summary, by combining \eqref{eq:Ser} with \cref{eq:ch1,eq:ch2,eq:ch3,eq:ch4,eq:ch5} we get
\begin{align*}
    &\iota\big(B_iB_{i+1}^2-(q+q^{-1})B_{i+1}B_iB_{i+1}+B_{i+1}^2B_i\big)=\sum_{t=0}^n-(q-q^{-1})^2P_{i,t}P_{i+1,0}P_{i+1,n}\\&+\sum_{m=1}^{n-1}\left((1-q^2)(1-q^4)P_{i,n-1}P_{i+1,1}+(1-q^{-2})(1-q^4)P_{i,n}P_{i+1,0}\right)P_{i+1,m}.
\end{align*}

Note that 
\begin{align*}
    P_{i,n-1}P_{i+1,1}&=q^{-3/2}\ren{X_{i,0}\cdots X_{i,n-1}X_{i+1,0}X_{i+1,1}}, \quad \text{and}\\ P_{i,n}P_{i+1,0}&=q^{1/2}\ren{X_{i,0}\cdots X_{i,n}X_{i+1,0}}.
\end{align*}
Since $X_{i,n}=X_{i+1,1}$, we deduce that $P_{i,n-1}P_{i+1,1}=q^{-2}P_{i,n}P_{i+1,0}$, which implies
\begin{equation*}
    (1-q^2)(1-q^4)P_{i,n-1}P_{i+1,1}+(1-q^{-2})(1-q^4)P_{i,n}P_{i+1,0}=0.
\end{equation*}

Hence we have
\begin{equation*}
    \iota\big(B_iB_{i+1}^2-(q+q^{-1})B_{i+1}B_iB_{i+1}+B_{i+1}^2B_i\big)=\sum_{t=0}^n-(q-q^{-1})^2P_{i,t}P_{i+1,0}P_{i+1,n}.
\end{equation*}

On the other hand, we have 
\begin{equation*}
    \iota(B_ik_{i+1})=-\sum_{t=0}^nP_{i,t}\ren{X_{i+1,0}^2X_{i+1,1}\cdots X_{i+1,n}}=-\sum_{t=0}^nP_{i,t}P_{i+1,0}P_{i+1,n}.
\end{equation*}
Therefore we deduce that
\begin{equation*}
    \iota\big(B_iB_{i+1}^2-(q+q^{-1})B_{i+1}B_iB_{i+1}+B_{i+1}^2B_i\big)=\iota(B_ik_{i+1}).
\end{equation*}

The verification for the equation 
\begin{equation*}
    \iota\left(B_{i+1}B_i^2-(q+q^{-1})B_iB_{i+1}B_i+B_i^2B_{i+1}\right)=\iota(B_{i+1}k_i)
\end{equation*}
is similar and will be skipped. We complete the proof.
\end{proof}

The next step to the proof of Theorem \ref{thm:emb} is to show the injectivity.

\begin{prop}\label{prop:inj}
    The algebra homomorphism $\iota$ is injective.
\end{prop}

\begin{proof}
    Fix the reduced expression $\mathbf{i}=(1,2,\cdots,n,1,2,\cdots,n-1,\cdots,1,2,1)$ of the longest element $w_0$. Recall the PBW-basis associated with $\mathbf{i}$ in \eqref{eq:PBWb}. 
    
    We set 
    \[
    w_{0,<k}=s_{i_1}s_{i_2}\cdots s_{i_{k-1}},\quad\text{for}\quad 1\leq k\leq r.
    \]
    Here we write $(i_1,i_2,\cdots,i_{r})=\mathbf{i}$, and $r=n(n+1)/2$. For each integer $k\in [1,r]$, take $k_1,k_2$, with $k_1\leq k_2$, such that 
    \[
    w_{0,<k}(\alpha_{i_k})=\alpha_{k_1}+\alpha_{k_1+1}+\cdots +\alpha_{k_2}.
    \]
    Here $\alpha_t$ ($1\leq t\leq n$) are simple roots in the root system. 

Recall the monomials $P_{i,k}$ in \eqref{eq:pik}. We \emph{claim} that for $1\leq k\leq r$, the element $\iota(B_{\mathbf{i}}(\mathbf{e}_k))$ has the leading term $$\ren{P_{k_1,0}P_{k_1+1,1}\cdots P_{k_2,k_2-k_1}},$$ when it is written as a polynomial in the quantum torus algebra. Namely, one has 
    \[
    \iota(B_{\mathbf{i}}(\mathbf{e}_k))=\ren{P_{k_1,0}P_{k_1+1,1}\cdots P_{k_2,k_2-k_1}}(1+R_k),
    \]
    where $R_k$ is a polynomial without constant term.

    We prove the claim by induction on the difference $k_2-k_1$. If $k_2=k_1$, by \cite{WZ22}*{Theorem 7.13} we have $B_{\mathbf{i}}(\mathbf{e}_k)=B_{k_1}$. The claim follows immediately. 
    
    Suppose $k_2>k_1$ and the claim holds for $k_2-k_1-1$. By our choice of the reduced expression, we have $s_{i_1}\cdots s_{i_{k-2}}(\alpha_{i_k})=\alpha_{k_2}$. Together with \emph{loc.cit.}, we have
    \begin{equation}\label{eq:bek}
        \begin{split}
            B_{\mathbf{i}}(\mathbf{e}_k)=&\bT_{i_1}\cdots \bT_{i_{k-1}}(B_{i_k})\\
            =&\bT_{i_1}\cdots \bT_{i_{k-2}}\big(\frac{q^{1/2}B_{i_{k-1}}B_{i_k}-q^{-1/2}B_{i_k}B_{i_{k-1}}}{q-q^{-1}}\big)\\
            =&\frac{q^{1/2}B_{\mathbf{i}}(\mathbf{e}_{k-1})B_{k_2}-q^{-1/2}B_{k_2}B_{\mathbf{i}}(\mathbf{e}_{k-1})}{q-q^{-1}}.
        \end{split}
    \end{equation}
    Also notice that
    \[
    w_{0,<k-1}(\alpha_{i_{k-1}})=\alpha_{k_1}+\alpha_{k_1+1}+\cdots+\alpha_{k_2-1}.
    \]
    By the induction hypothesis, element $\iota(B_{\mathbf{i}}(\mathbf{e}_{k-1}))$ has the leading term 
    \[
    \ren{P_{k_1,0}P_{k_1+1,1}\cdots P_{k_2-1,k_2-k_1-1}}.
    \]

    Recall from \eqref{eq:em1} that
    \[
    \iota(B_{k_2})=\sum_{t=0}^nP_{k_2,t}.
    \]

    Notice that for $0\leq t\leq k_2-k_1-1$ we have
    \[
    P_{k_1,0}\cdots P_{k_2-1,k_2-k_1-1}P_{k_2,t}=q^{-1}P_{k_2,t}P_{k_1,0}\cdots P_{k_2-1,k_2-k_1-1}.
    \]
    We also have
    \[
    P_{k_1,0}\cdots P_{k_2-1,k_2-k_1-1}P_{k_2,k_2-k_1}=qP_{k_2,k_2-k_1}P_{k_1,0}\cdots P_{k_2-1,k_2-k_1-1}.
    \]

It is direct verify that for any two elements $f$, $g$ in the quantum torus algebra, one has
    \[  
    \frac{q^{1/2}fg-q^{-1/2}gf}{q-q^{-1}}=\begin{cases}
        \ren{fg}, &\text{if }fg=qgf,\\
        0, &\text{if }fg=q^{-1}gf.
    \end{cases}
    \]
Plug into \eqref{eq:bek}, we deduce that $\iota(B_{\mathbf{i}}(\mathbf{e}_k))$ has the leading term $$\ren{P_{k_1,0}P_{k_1+1,1}\cdots P_{k_2,k_2-k_1}}.$$ This completes the proof of the claim. 

It follows from the claim that images of $B_{\mathbf{i}}(\mathbf{a})$, for $\mathbf{a}\in\mathbb{N}^{r}$, under $\iota$ have different leading terms. 

Finally, since the leading term of $\iota(B_{\mathbf{i}}(\mathbf{a}))$ does not contain variables $X_{t,t}$, for $1\leq t\leq n$, the degree of $X_{t,t}$ in the leading term of $\iota(B_{\mathbf{i}}(\mathbf{a})k(\mathbf{b}))$ can only be contributed by $\iota(k_t)$. Hence the images of PBW-basis elements
\[
\{B_\mathbf{i}(\mathbf{a})k(\mathbf{b})\mid \mathbf{a}\in\mathbb{N}^{r},\mathbf{b}\in\mathbb{Z}^n\}
\]
under $\iota$ also have different leading terms. Therefore they are linearly independent in the quantum torus algebra. This completes the proof of the injectivity.
\end{proof}

The final step to the proof of Theorem \ref{thm:emb} is to show that elements in the image of the algebra $\tUi$ remain to be Laurent polynomials after mutations. 

\begin{prop}\label{prop:qcl}
Elements in the image of $\iota$ belong to $\FOX$.
\end{prop}

\begin{proof}
It will suffice to show that generators $\iota(B_i)$ and $\iota(k_i)$, for $1\leq i\leq n$, belong to the quantum cluster algebra $\mathcal{O}_q(\mathcal{X}_{|\Sigma_n|})$.

Since $\iota(k_i)$ ($1\leq i\leq n$) are central monomials in $\T_{\Sigma_n}$, they belong to $\mathcal{O}_q(\mathcal{X}_{|\Sigma_n|})$ by \cite{GS22}*{\S 18.0.4}. We next show that $\iota(B_i)\in\mathcal{O}_q(\mathcal{X}_{|\Sigma_n|})$, for $1\leq i\leq n$. 

In the setting of \emph{quantum cluster $\mathcal{A}$-algebras}, it is well-known that a Laurent polynomial in one cluster chart belongs to the quantum cluster algebra if it remains to be a Laurent polynomial after one-step mutations (see \cite{BZ05}*{Theorem 5.1}). This criteria remains to be true for quantum cluster $\mathcal{X}$-algebras, by applying a trick of Goncharov--Shen. Let us briefly mention their arguments.

Following \cite{GS22}*{Proposition 18.5} (see also \cite{SS19}*{Proposition 4.10}), we consider the framing $\Sigma_n'$ of the quiver $\Sigma_n$. (See \emph{loc.cit.} for precise definitions.) One has the composition of algebra homomorphisms
\begin{equation*}
    \kappa^*:\mathcal{O}_q(\mathcal{X}_{|\Sigma_n|})\hookrightarrow\mathcal{O}_q(\mathcal{X}_{|\Sigma_n'|})\overset{\sim}{\longrightarrow}\mathcal{O}_q(\mathcal{A}_{|\Sigma_n'|}),
\end{equation*}
where $\mathcal{O}_q(\mathcal{A}_{|\Sigma_n'|})$ is the corresponding Berenstein--Zelevinsky quantum upper cluster algebra, and the isomorphism is induced by the cluster ensemble map (cf. \cite{FG09}). Since the map $\kappa^*$ is moreover compatible with cluster mutations, the one-step mutation criteria transfers to quantum cluster $\mathcal{X}$-algebras naturally. 

The remaining part of the proof is dedicated to verify that elements $\iota(B_i)$, for $1\leq i\leq n$, remain to be Laurent polynomials after one-step mutations. The verification is by direct computation.   

The case when $n=1,2$ is direct. We assume $n>2$. Let us write $\mu_{i,k}$ to be the mutation at the vertex $X_{i,k}$.

We compute $\mu_{i,1}(\iota(B_i))$. It is easy to see that
\begin{equation*}
    \mu_{i,1}(X_{i,k})=\begin{cases}
        X_{i,0}(1+qX_{i,1}^{-1})^{-1}, &\text{if }k=0,\\
        X_{i,1}^{-1}, &\text{if }k=1,\\
        X_{i,k}(1+qX_{i,1}), &\text{if }k=2\text{ or
        }n,\\
        X_{i,k}, &\text{if }3\leq k< n.
    \end{cases}
\end{equation*}
Hence 
\begin{equation*}
    \mu_{i,1}(P_{i,0}+P_{i,1})=\mu_{i,1}(X_{i,0}(1+qX_{i,1}))=X_{i,0},\;\text{and}\; \mu_{i,1}(X_{i,0}X_{i,1}X_{i,2})=q^{-1}X_{i,0}X_{i,2}.
\end{equation*}
Therefore we have
\begin{equation*}
\begin{split}
    \mu_{i,1}(\iota(B_i))
        =&X_{i,0}+\sum_{k=2}^{n-1}q^{k-1}X_{i,0}X_{i,2}\cdots X_{i,k}+q^{n-1}X_{i,0}X_{i,2}\cdots X_{i,n}(1+qX_{i,1}),
\end{split}
\end{equation*}
which is a Laurent polynomial.

We next compute $\mu_{i,t}(\iota(B_i))$, for $1<t<n$. It is direct to see that
\begin{equation}\label{eq:lau}
    \mu_{i,t}(X_{i,k})=\begin{cases}
        X_{i,t-1}(1+qX_{i,t}^{-1})^{-1}, &\text{if }k=t-1,\\
        X_{i,t}^{-1}, &\text{if }k=t,\\
         X_{i,t+1}(1+qX_{i,t}), &\text{if }k=t+1,\\
         X_{i,k}, &\text{if otherwise}.
    \end{cases}
\end{equation}
Hence 
\begin{equation}\label{eq:sq}
\mu_{i,t}(X_{i,t-1}(1+qX_{i,t}))=X_{i,t-1},\;\text{and}\;\mu_{i,1}(X_{i,t-1}X_{i,t}X_{i,t+1})=q^{-1}X_{i,t-1}X_{i,t+1}.
\end{equation}

Express $\iota(B_i)$ as the following,
\begin{equation*}
    \begin{split}
        \iota(B_i)
        =&\sum_{k=0}^{t-2}P_{i,k}+q^{t-1}X_{i,0}\cdots X_{i,t-2}X_{i,t-1}(1+qX_{i,t})\\&+\sum_{k=t+1}^nq^kX_{i,0}\cdots X_{i,t-1}X_{i,t}X_{i,t+1}\cdots X_{i,k}.
        \end{split}
\end{equation*}
Then thanks to \eqref{eq:lau} and \eqref{eq:sq}, we deduce that $\mu_{i,t}(\iota(B_i))$ is a Laurent polynomial.

We next compute $\mu_{i,n}(\iota(B_i))$. It is direct to see that
\begin{equation}\label{eq:lau1}
    \mu_{i,n}(X_{i,k})=\begin{cases}
        X_{i,0}(1+qX_{i,n}), &\text{if }k=0,\\
        X_{i,k}(1+qX_{i,n}^{-1})^{-1}, &\text{if }k=1\text{ or }n-1,\\
        X_{i,n}^{-1}, &\text{if }k=n,\\
        X_{i,k}, &\text{if otherwise}.
    \end{cases}
\end{equation}

Then 
\begin{equation}\label{eq:qw}
    \mu_{i,n}(X_{i,0}X_{i,1})=qX_{i,0}X_{i,n}X_{i,1},\;\text{and}\;\mu_{i,n}(X_{i,n-1}(1+qX_{i,n}))=X_{i,n-1}.
\end{equation}

We express $\iota(B_i)$ as the following,
\begin{equation}\label{eq:qe}
    \iota(B_i)=X_{i,0}+\sum_{k=1}^{n-2}q^kX_{i,0}X_{i,1}\cdots X_{i,k}+q^{n-1}X_{i,0}X_{i,1}\cdots X_{i,n-1}(1+qX_{i,n}).
\end{equation}
Thanks to \eqref{eq:lau1} and \eqref{eq:qw}, each term on the right hand side of \eqref{eq:qe} is a Laurent polynomial under the mutation $\mu_{i,n}$. Therefore $\mu_{i,n}(\iota(B_i))$ is a Laurent polynomial.

Now we consider the mutation at the unfrozen vertex $Y$, where $Y\notin\{X_{i,k}\mid 1\leq k\leq n\}$. Let $\sigma_{n,i}$ to be the subquiver of $\Sigma_{n}$ containing vertices $X_{i,k}$ ($0\leq k\leq n$). If $Y$ is not connected to $\sigma_{n,i}$, then the expression of $\iota(B_i)$ remains the same after the mutation $\mu_Y$. Suppose $Y$ is connected to $\sigma_{n,i}$. Then from the quiver one sees that $Y$ is connected to either 2 or 4 vertices of $\sigma_{n,i}$.

Assume $Y$ is connected to 2 vertices of $\sigma_{n,i}$, say $X_{i,t}$ and $X_{i,t+1}$, for $1\leq t<n$. Then
\begin{equation*}
    \mu_Y(X_{i,k})=\begin{cases}
        X_{i,t}(1+qY), &\text{if }k=t,\\
        X_{i,t+1}(1+qY^{-1})^{-1}, &\text{if }k=t+1,\\
        X_{i,t}, &\text{if otherwise.}
    \end{cases}
\end{equation*}
Hence
\begin{equation*}
    \mu_Y(X_{i,t}X_{i,t+1})=qX_{i,t}YX_{i,t+1}.
\end{equation*}
It follows easily that under mutation $\mu_Y$, the monomial $P_{i,k}$ ($0\leq k\leq n$) remain to be Laurent, so $\mu_Y(\iota(B_i))$ is a Laurent polynomial.

Assume $Y$ is connected to 4 vertices of $\sigma_{n,i}$, in which case $Y$ is connected to $X_{i,1}$, $X_{i,2}$, $X_{i,n-1}$ and $X_{i,n}$. Then
\begin{equation*}
    \mu_Y(X_{i,k})=\begin{cases}
        X_{i,k}(1+qY), &\text{if }k=1\text{ or }n-1,\\
        X_{i,k}(1+qY^{-1})^{-1}, &\text{if }k=2\text{ or }n,\\
        X_{i,k}, &\text{if otherwise}.
    \end{cases}
\end{equation*}
Hence 
\begin{equation*}
    \mu_Y(X_{i,1}X_{i,2})=qX_{i,1}YX_{i,2},\quad\text{and}\quad \mu_Y(X_{i,n-1}X_{i,n})=qX_{i,n-1}YX_{i,n}.
\end{equation*}
It then follows easily that elements $\mu_Y(P_{i,k})$, for $0\leq k\leq n$, remain to be Laurent polynomials, so $\mu_Y(\iota(B_i))$ is a Laurent polynomial.
\end{proof}

\begin{remark}
    In the cluster realisations of quantum groups, it is known that Chevalley generators $E_i$, $F_i$ become monomials after certain sequences of cluster mutations. This is not true for generators $B_i$ in our setting.
\end{remark}

Recall the central element $C_i$ in \eqref{eq:cen}. Central reductions of $\imath$quantum groups (see \eqref{eq:icen}) give the cluster realisations the algebra $\dUi$.
\begin{cor}\label{cor:fui}
    The embedding $\iota$ induces an embedding of the $\imath$quantum group $\dUi$ into the quotient of the algebra $\FOX$ by the relations $C_i=1$, for all $1\leq i\leq n$.
\end{cor}

\subsection{Coideal structures}\label{sec:coid}

In this subsection, we interpret coideal structures of $\imath$quantum groups in terms of the amalgamation of quivers.

Recall the coideal structure $\Delta: \tUi\longrightarrow \tUi\otimes \tU$ of $\imath$quantum groups in \eqref{eq:cop}, and the labelling of the quivers $\Sigma_n$, $\mathcal{D}_n$ in Section \ref{sec:quiv}, Section \ref{sec:clq}, respectively. 

Let us consider the amalgamation of the quivers $\Sigma_n$ and $\mathcal{D}_n$, by identifying vertices $V_i$ with $X_{i,0}$, for $1\leq i\leq n$. The resulting quiver is denoted by $\mathcal{Z}_n$. Then we have an algebra embedding $\T_{\mathcal{Z}_n}\hookrightarrow\T_{\Sigma_n}\otimes \T_{\mathcal{D}_n}$ (see \eqref{eq:aml}).

Recall the cluster realisations $\iota$ and $t$ of $\imath$quantum groups and quantum groups, in Theorem \ref{thm:tqg} and Theorem \ref{thm:emb}, respectively. The algebra homomorphism $(\iota\otimes t)\circ\Delta$ factors through $\T_{\mathcal{Z}_n}$,
\begin{equation}\label{eq:itd}
    (\iota\otimes t)\circ\Delta: \tUi\longrightarrow \T_{\mathcal{Z}_n}\subset \T_{\Sigma_n}\otimes\T_{\mathcal{D}_n}.
\end{equation}

For each $1\leq i\leq n$, we consider a path $\mathbf{L}_i$ in $\mathcal{Z}_n$, which is the concatenation of three paths, $\mathbf{\Lambda}_i$, $\mathbf{O}_i$ and $\mathbf{V}_i$. Here $\mathbf{\Lambda}_i$ is the path in $\mathcal{D}_n$, which starts from the vertex $\Lambda_{n+1-i}$, going in the North-West direction first and then in the South-West direction, and ends at the vertex $V_i$. Similarly $\mathbf{V}_i$ is the path in $\mathcal{D}_n$. It starts from $V_i$ and ends at $\Lambda_{n+1-i}$, which goes in the South-East direction first, and in the North-East direction next. The path $\mathbf{O}_i$ is the loop in $\Sigma_n$, which passes vertices $(X_{i,0},X_{i,1},\cdots,X_{i,n},X_{i,0})$. Then the path $\mathbf{L}_i$ is a loop, starting and ending at the same vertex $\Lambda_{n+1-i}$. Besides the starting and ending point, $\mathbf{L}_i$ self-intersects at another vertex $X_{i,0}=V_i$. We label the vertices $Z_{i,1},Z_{i,2},\cdots,Z_{i,p}=Z_{i,1}$ along the path $\mathbf{L}_i$, where $p=3n+4$. Note that $X_{i,0}$ and $\Lambda_{n+1-i}$ are labelled twice.

\begin{figure}[h]
\centering
\begin{tikzpicture}[every node/.style={inner sep=0, minimum size=0.4cm, draw, thick}, thick, y=0.8cm]
\begin{scope}[>=latex,xshift=1]
\node (8) at (-2,-1) [circle] {\tiny{8}};
\node (9) at (0,-3) [circle] {\tiny{9}};
\node (4) at (2,-1) {\tiny{4}};
\node (3) at (-2,1) [circle] {\tiny{3}};
\node (7) at (-1,0) [circle] {\tiny{7}};
\node (10) at (0,-1) [circle] {\tiny{10}};
\node (5) at (1,0) [circle] {\tiny{5}};
\node (1) at (2,1) {\tiny{1}};
\node (2) at (0,3) [circle] {\tiny{2}};
\node (6) at (0,1) [circle] {\tiny{6}};
\node (11) at (-3,0) [circle] {\tiny{11}};
\node (12) at (-4,1.5) [circle] {\tiny{12}};
\node (13) at (-4,-1.5) [circle] {\tiny{13}};

\draw [->] (5) to (6);
\draw [->] (6) to (7);
\draw [->] (7) to (10);
\draw [->] (10) to (5);

\draw [->] (1) to (2);
\draw [->] (2) to (3);
\draw [->] (3) to (7);
\draw [->] (7) to (8);
\draw [->] (8) to (9);
\draw [->] (9) to (4);
\draw [->] (4) to (5);
\draw [->] (5) to (1);

\draw [->] (5) to (9);
\draw [->] (9) to (7);
\draw [->] (7) to (2);
\draw [->] (2) to (5);

\draw [->] (11) to (3);
\draw [->] (8) to (11);
\draw [->] (3) to (12);
\draw [->] (13) to (8);

\draw [->,double equal sign distance] (12) to (11);
\draw [->,double equal sign distance] (11) to (13);
\draw [->,double equal sign distance] (13) to (12);

\draw [->, dashed] (1) to (4);
\end{scope}

\end{tikzpicture}
\caption{$\mathcal{Z}_2$-quiver.}
\label{fig:Z2}
\end{figure}
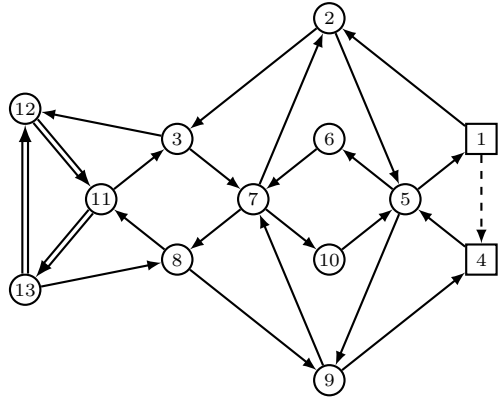

For example, when $n=2$, the $\mathcal{Z}_2$-quiver is as shown in Figure \ref{fig:Z2}. Then $\mathbf{L}_1$ is the path consisting of vertices labelled by $(1,2,3,12,11,3,7,10,5,1)$, and $\mathbf{L}_2$ is the path consisting of vertices labelled by $(4,5,6,7,8,11,13,8,9,4)$.

The following proposition can be verified directly.

\begin{prop}\label{prop:cop}
    Under the map \eqref{eq:itd}, the images of $B_i$ and $k_i$, for $1\leq i\leq n$, in $\T_{\mathcal{Z}_n}$ are 
\begin{equation}
    \Delta(B_i)=\sum_{l=1}^{p-1}\ren{Z_{i,1}Z_{i,2}\cdots Z_{i,l}},\quad \Delta(k_i)=-\ren{Z_{i,1}Z_{i,2}\cdots Z_{i,p}}.
\end{equation}
\end{prop}

\subsection{Quasi-cluster nature of the braid group actions}\label{sec:qucl}

In this subsection, we explicitly express braid group symmetries on $\imath$quantum groups in terms of compositions of cluster mutations, permutations of cluster variables, and renormalizations of frozen variables, under the embedding in Theorem \ref{thm:emb}.

For $1\leq i\leq n$ and $n\geq 2$, consider the following sequence of mutations on any quiver mutation-equivalent to $\Sigma_n$,
\begin{equation*}
    \widetilde{\mu}_i=\mu_{i,2}\circ\mu_{i,3}\circ\cdots\circ\mu_{i,n-1}\circ\mu_{i,n}\circ\mu_{i,n-1}\circ\cdots \circ\mu_{i,3}\circ\mu_{i,2}.
\end{equation*}
For $n=1$, $\widetilde{\mu}_i$ is defined to be the identity. We write $\Sigma_{n,i}=\widetilde{\mu}_i\Sigma_n$. Then $\widetilde{\mu}_i\Sigma_{n,i}=\Sigma_n$. Let $S_i$ be the map between quivers which are mutation-equivalent to $\Sigma_n$, given by swapping vertices $X_{i,1}$ and $X_{i,n}$. It induces a map between corresponding quantum torus algebras, which is still denoted by $S_i$. Then we have algebra automorphism between fraction fields 
$$S_i\circ\widetilde{\mu_i}:\tT_{\Sigma_n}\longrightarrow \tT_{S_i\Sigma_{n,i}}.$$
Here by slightly abuse of notations we still use $\widetilde{\mu}_i$ to denote the composition of the corresponding sequence of quantum cluster mutations. 

For $n\geq 2$, we denote $\Sigma_{n,i}^k$ ($0\leq k\leq 2n-3$) to be the quiver obtained by applying the first $(2n-3-k)$-step mutations of $\widetilde{\mu}_i$. Then $\Sigma_{n,i}^{2n-3}=\Sigma_n$ and $\Sigma_{n,i}=\Sigma_{n,i}^{0}$.

The quiver $S_i\Sigma_{n,i}$ can be obtained from $\Sigma_n$, by removing the arrows $X_{i,1}\rightarrow X_{i-1,0}$, $X_{i+1,0}\rightarrow X_{i+1,1}$, $X_{i+1,n}\rightarrow X_{i+1,0}$, adding the arrows $X_{i+1,1}\rightarrow X_{i-1,0}$, $X_{i,n-1}\rightarrow X_{i+1,0}$, and reversing the dashed arrow $X_{i,0}\dasharrow X_{i+1,0}$. In particular, $S_i\Sigma_{n,i}$ has the same unfrozen part as $\Sigma_n$ (cf. \cite{CS23}*{5.2}).

Recall the central element $C_i$ in \eqref{eq:cen}. The following lemma follows directly from the shape of the quiver $S_i\Sigma_{n,i}$. 

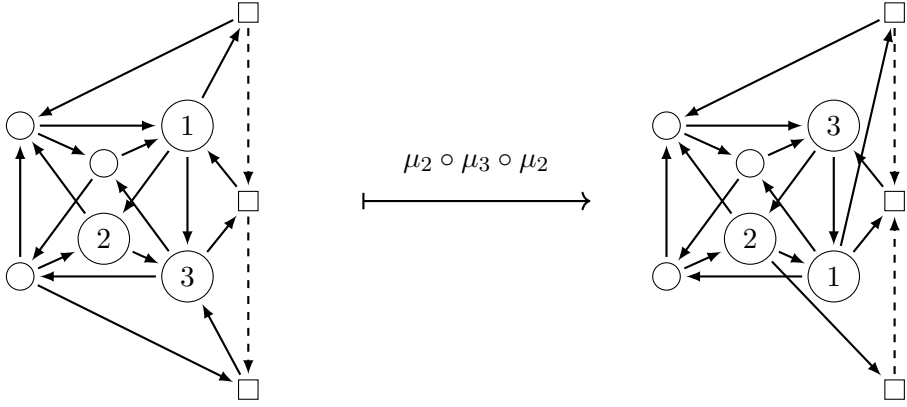
\begin{figure}[h]
    \centering
    \begin{tikzpicture}every node/.style={inner sep=0, minimum size=0.4cm, thick, draw, fill=white}, thick, x=0.9cm, y=1.2cm
    
    \node (9) at (1.5, 0.5) [draw, circle] {};
    \node (2) at (1.5, 2.5) [draw, circle] {};
    \node (4) at (3.7, 2.5) [draw, circle] {$\tiny{1}$};
    \node (7) at (3.7, 0.5) [draw, circle] {$\tiny{3}$};
    \node (6) at (2.6, 1) [draw, circle] {$\tiny{2}$};
    \node (3) at (2.6, 2) [draw, circle] {};
    \node (1) at (4.5, 4) [draw] {};
    \node (5) at (4.5, 1.5) [draw] {};
    \node (8) at (4.5, -1) [draw] {};

    \begin{scope}[>=latex]
        \qdarrow{1}{5};
        \qdarrow{5}{8};
        \qarrow{1}{2};
        \qarrow{2}{3};
        \qarrow{3}{4};
        \qarrow{4}{1};
        \qarrow{5}{4};
        \qarrow{4}{6};
        \qarrow{6}{7};
        \qarrow{7}{5};
        \qarrow{8}{7};
        \qarrow{7}{3};
        \qarrow{3}{9};
        \qarrow{9}{8};
        \qarrow{2}{4};
        \qarrow{4}{7};
        \qarrow{7}{9};
        \qarrow{9}{2};
        \qarrow{6}{2};
        \qarrow{9}{6};
    \end{scope}

    \draw[thick,|->] (6,1.5) -- (9,1.5);
    \node at (7.5,2) {$\mu_{2}\circ\mu_{3}\circ\mu_{2}$};

    \node (9) at (10, 0.5) [draw, circle] {};
    \node (2) at (10, 2.5) [draw, circle] {};
    \node (4) at (12.2, 2.5) [draw, circle] {$\tiny{3}$};
    \node (7) at (12.2, 0.5) [draw, circle] {$\tiny{1}$};
    \node (6) at (11.1, 1) [draw, circle] {$\tiny{2}$};
    \node (3) at (11.1, 2) [draw, circle] {};
    \node (1) at (13, 4) [draw] {};
    \node (5) at (13, 1.5) [draw] {};
    \node (8) at (13, -1) [draw] {};

    \begin{scope}[>=latex]
        \qdarrow{1}{5};
        \qdarrow{8}{5};
        \qarrow{1}{2};
        \qarrow{2}{3};
        \qarrow{3}{4};
        \qarrow{5}{4};
        \qarrow{4}{6};
        \qarrow{6}{7};
        \qarrow{7}{5};
        \qarrow{7}{3};
        \qarrow{3}{9};
        \qarrow{2}{4};
        \qarrow{4}{7};
        \qarrow{7}{9};
        \qarrow{9}{2};
        \qarrow{6}{2};
        \qarrow{9}{6};
        \qarrow{7}{1};
        \qarrow{6}{8};
    \end{scope}
    \end{tikzpicture}
    \caption{Braid group action $\bT_2$ for $n=3$}
    \label{fig:mu}
\end{figure}

\begin{lemma}
    There is a well-defined algebra automorphism $M_i:\tT_{S_i\Sigma_{n,i}}\rightarrow \tT_{\Sigma_n}$ between fractional fields of quantum torus algebras, given by
    \begin{equation*}
        X_{t,k}\mapsto\begin{cases}
            X_{i,0}C_i^{-1}, &\text{if }(t,k)=(i,0),\\
            q^{1/2}X_{i-1,0}X_{i,0}, &\text{if }(t,k)=(i-1,0),\\
            q^{1/2}X_{i,0}X_{i+1,0}X_{i+1,1}, &\text{if }(t,k)=(i+1,0),\\
            X_{t,k}, &\text{if otherwise.}
        \end{cases}
    \end{equation*}
\end{lemma}

For $1\leq i\leq n$, define the automorphism
    $$\cT_i=M_i\circ S_i\circ\widetilde{\mu}_i:\tT_{\Sigma_n}\longrightarrow\tT_{\Sigma_n}.$$

Recall the embedding $\iota$ in Theorem \ref{thm:emb}, and the braid group symmetries $\bT_i$ in Section \ref{sec:braid}.

\begin{theorem}\label{thm:braid}
For $1\leq i\leq n$, the automorphism $\cT_i$ restricts to $\tUi$ under the embedding $\iota$. Moreover it coincides with $T_i$ as defined in \eqref{braid}, when restricting to $\tUi$.
\end{theorem}

The classical counterpart of Theorem \ref{thm:braid} was obtained in \cite{CS23}*{Theorem 5.6}. 

Let us firstly compute the action of $\cT_i$ on generators $X_{t,k}$ of the quantum torus algebra. For any $1\leq i\leq n$ and $0\leq k\leq n$, we set
    \begin{equation}\label{eq:W}
     W_{i,k}=\sum_{t=0}^kq^tX_{i,0}X_{i,1}\cdots X_{i,t}.
    \end{equation}

The following lemma is the quantum analogue of \cite{CS23}*{Lemma 5.5}. 

\begin{lemma}\label{le:T}
    Suppose $n\geq 3$. For $1\leq i\leq n$, we have
    \begin{equation}\label{eq:T1}
        \cT_i(X_{i,k})=\begin{cases}
            W_{i,n-1}C_i^{-1}, &\text{if }k=0,\\
            q^{-1}W_{i,n-1}^{-1}X_{i,0}^{-1}(1+qX_{i,1}^{-1})^{-1}C_i, &\text{if }k=1,\\
            W_{i,k}^{-1}X_{i,k}W_{i,k-2}=W_{i,k}^{-1}W_{i,k-2}X_{i,k}, &\text{if }2\leq k\leq n-1,\\
            X_{i,0}X_{i,n}W_{i,n-2}C_i^{-1},& \text{if }k=n.
        \end{cases}
    \end{equation}
    For $1<i\leq n$, we have
    \begin{equation}\label{eq:T2}
        \cT_i(X_{i-1,k})=\begin{cases}
            q^{1/2}X_{i-1,0}X_{i,0}, &\text{if }k=0,\\
            W_{i,k-1}^{-1}X_{i-1,k}W_{i,k}=X_{i-1,k}W_{i,k-1}^{-1}W_{i,k}, &\text{if }1\leq k\leq n-2,\\
            (1+q^{-1}X_{i,1})X_{i-1,n-1}W_{i,n-2}^{-1}W_{i,n-1}, &\text{if }k=n-1,\\
            q^{-1}W_{i,n-1}^{-1}X_{i,0}^{-1}(1+qX_{i,1}^{-1})^{-1}C_i,&\text{if }k=n.
        \end{cases}
    \end{equation}
    For $1\leq i<n$, we have
    \begin{equation}\label{eq:T3}
        \cT_i(X_{i+1,k})=\begin{cases}
           q^{-1/2}W_{i,n-1}^{-1}X_{i+1,0}C_i, &\text{if }k=0,\\
           X_{i,0}X_{i+1,1}W_{i,n-2}C_i^{-1},&\text{if }k=1,\\
           (1+q^{-1}X_{i,1})X_{i+1,2}W_{i,n-2}^{-1}W_{i,n-1},&\text{if }k=2,\\
            W_{i,k-2}^{-1}X_{i+1,k}W_{i,k-1}=X_{i+1,k}W_{i,k-2}^{-1}W_{i,k-1}, &\text{if }3\leq k\leq n.
        \end{cases}
    \end{equation}
And $\cT_i$ fixes $X_{j,k}$, if $X_{j,k}$ is not one of the above generators.
\end{lemma}

\begin{proof}
We start by introducing some polynomials in the quantum torus algebra, which play important role in our computation. 

Fix an integer $i\in [1,n]$. For $2\leq k<n$, define the element
\begin{equation}
    A_k=A_{k,i}=1+qX_{i,n}+q^2X_{i,n}X_{i,2}+\cdots+q^kX_{i,n}X_{i,2}\cdots X_{i,k}\quad \text{in }\tT_{\Sigma_{n,i}}.
\end{equation}
By direct computation, one has
\begin{equation}
    X_k^{-1}+\cdots+q^{k-2}X_2^{-1}\cdots X_{k}^{-1}=q(A_k-A_{k-1})^{-1}(A_{k-1}-1),
\end{equation}
and
\begin{equation}\label{eq:ak1}
    (A_{k-1}-1)(A_k-A_{k-1})=q^{-2}(A_k-A_{k-1})(A_{k-1}-1).
\end{equation}

We firstly compute $\cT_i(X_{i,k})$, for $2\leq k\leq n-1$. 

Consider the $(k-1)$-step of mutations $\mu_{i,2}\cdots \mu_{i,k}:\tT_{\Sigma_{n,i}^{k-1}}\longrightarrow \tT_{\Sigma_{n,i}}.$ We have
\begin{equation}\label{eq:mukn}
    \begin{split}
         \mu_{i,2}\cdots\mu_{i,k}(X_{i,n})&=\mu_{i,2}\cdots\mu_{i,k-1}(X_{i,n}(1+qX_{i,k}))\\
         &=\mu_{i,2}\cdots\mu_{i,k-2}(X_{i,n}(1+qX_{i,k-1}(1+qX_{i,k}(1+qX_{i,k-1}^{-1})^{-1})))\\
         &=\mu_{i,2}\cdots\mu_{i,k-2}(X_{i,n}(1+qX_{i,k-1}+q^2X_{i,k-1}X_{i,k}))\\
         &=\cdots\cdots\\
         &=X_{i,n}(1+qX_{i,2}+q^2X_{i,2}X_{i,3}+\cdots+q^{k-1}X_{i,2}\cdots X_{i,k})\\
         &=q^{-1}(A_k-1).
    \end{split}
\end{equation}

We also have
\begin{equation}\label{eq:mukk}
    \begin{split}
        \mu_{i,2}\cdots\mu_{i,k}(X_{i,k})&=\mu_{i,2}\cdots\mu_{i,k-1}(X_{i,k}^{-1})\\
        &=\mu_{i,2}\cdots\mu_{i,k-2}((X_{i,k}(1+qX_{i,k-1}^{-1})^{-1})^{-1})\\
        &=\mu_{i,2}\cdots\mu_{i,k-2}(X_{i,k}^{-1}+qX_{i,k-1}^{-1}X_{i,k}^{-1})\\
        &=\cdots\cdots\\
        &=X_{i,k}^{-1}+qX_{i,k-1}^{-1}X_{i,k}^{-1}+\cdots+q^{k-2}X_{i,2}^{-1}\cdots X_{i,k}^{-1}\\
        &=q(A_k-A_{k-1})^{-1}(A_{k-1}-1).
    \end{split}
\end{equation}

Consider the $(2n-k-2)$-step mutations
$\mu_{i,2}\cdots\mu_{i,n}\cdots\mu_{i,k+1}:\tT_{\Sigma_{n,i}^{2n-k-2}}\longrightarrow\tT_{\Sigma_{n,i}}.$ We have
\begin{equation}\label{eq:muk1k}
    \begin{split}
        &\mu_{i,2}\cdots\mu_{i,n}\cdots\mu_{i,k+1}(X_{i,k})\\=&\mu_{i,2}\cdots\mu_{i,n}\cdots\mu_{i,k+2}(X_{i,k}(1+qX_{i,k+1}))\\=&\cdots\cdots\\
        =&\mu_{i,2}\cdots\mu_{i,n-1}(X_{i,k}(1+\cdots(1+qX_{i,n-2}(1+qX_{i,n-1}(1+qX_{i,n})))\cdots ))\\
        =&\mu_{i,2}\cdots\mu_{i,n-2}(X_{i,k}(1+\cdots\\&(1+qX_{i,n-2}(1+qX_{i,n-1}^{-1})^{-1}(1+qX_{i,n-1}^{-1}(1+qX_{i,n}(1+qX_{i,n-1})))\cdots))\\
        =&\mu_{i,2}\cdots\mu_{i,n-2}(X_{i,k}(1+\cdots(1+qX_{i,n-2}(1+qX_{i,n}))\cdots))\\
        =&\cdots\cdots\\
        =&\mu_{i,2}\cdots\mu_{i,k}(X_{i,k}(1+qX_{i,n}))
        \\=&q(A_k-A_{k-1})^{-1}(A_{k-1}-1)A_k,
    \end{split}
\end{equation}
where the last equality follows from \cref{eq:mukk,eq:mukn}.

Then we have
\begin{equation}\label{eq:muik}
    \begin{split}
        \widetilde{\mu}_i(X_{i,k})&=\mu_{i,2}\cdots\mu_{i,n}\cdots\mu_{i,2}(X_{i,k})\\
        &=\mu_{i,2}\cdots\mu_{i,n}\cdots\mu_{i,k-1}(X_{i,k})\\
        &=\mu_{i,2}\cdots\mu_{i,n}\cdots\mu_{i,k}(X_{i,k}(1+qX_{i,k-1}))\\
        &=\mu_{i,2}\cdots\mu_{i,n}\cdots \mu_{i,k+1}(X_{i,k})^{-1} (1+q\mu_{i,2}\cdots\mu_{i,n}\cdots\mu_{i,k}(X_{i,k-1}))\\
        &\overset{(\heartsuit1)}{=}q^{-1}A_k^{-1}(A_{k-1}-1)^{-1}(A_k-A_{k-1})(1+q^2(A_{k-1}-A_{k-2})^{-1}(A_{k-2}-1)A_{k-1})\\
        &\overset{(\heartsuit2)}{=}q^{-1}A_k^{-1}(A_{k-1}-1)^{-1}(A_k-A_{k-1})(A_{k-1}-A_{k-2})^{-1}(A_{k-1}-A_{k-2}+q^2(A_{k-2}-1)A_{k-1})\\
        &\overset{(\heartsuit3)}{=}q^{-2}A_k^{-1}(A_{k-1}-1)^{-1}X_k((A_{k-1}-A_{k-2})A_{k-2}+q^2(A_{k-2}-1)A_{k-2})\\
        &=A_k^{-1}(A_{k-1}-1)^{-1}((A_{k-1}-A_{k-2})A_{k-2}+(A_{k-2}-1)A_{k-2})X_k\\
        &=A_k^{-1}A_{k-2}X_{i,k}.
    \end{split}
\end{equation}
Here $(\heartsuit1)$ follows from \eqref{eq:muk1k}, $(\heartsuit2)$ follows from \eqref{eq:ak1}, and $(\heartsuit3)$ follows from $X_k(A_{k-1}-A_{k-2})=q^2(A_{k-1}-A_{k-2})X_k$ and $X_kA_{k-2}=A_{k-2}X_k$.



It then follows that 
\begin{equation}
    \cT_i(X_{i,k})=M_i\circ S_i\circ \widetilde{\mu}_i(X_{i,k})=W_{i,k}^{-1}X_{i,k}W_{i,k-2}=W_{i,k}^{-1}W_{i,k-2}X_{i,k}.
\end{equation}
The last equality is obtained by noticing that $X_{i,k}$ commutes with every term in the summation of $W_{i,k-2}$.

We next compute $\cT_i(X_{i,0})$. We have
\begin{equation*}
    \begin{split}
        \widetilde{\mu}_i(X_{i,0})&=\mu_{i,2}\cdots \mu_{i,n}\cdots\mu_{i,2}(X_{i,0})\\
        &=\mu_{i,2}\cdots\mu_{i,n-1}(X_{i,0}(1+qX_{i,n}))\\
        &=X_{i,0}(1+qX_{i,n}+q^2X_{i,n}X_{i,2}+\cdots +q^{n-1}X_{i,n}X_{i,2}\cdots X_{i,n-1}).
    \end{split}
\end{equation*}
Here the last equality follows from \eqref{eq:mukn} and the equality
$$\mu_{i,2}\cdots\mu_{i,n-1}(X_{i,0})=X_{i,0}.$$
Hence we have
\begin{equation}
\begin{split}
    \cT_i(X_{i,0})&=M_i\circ S_i\circ\widetilde{\mu}_i(X_{i,0})\\&=M_i(X_{i,0}(1+qX_{i,n}+q^2X_{i,1}X_{i,2}+\cdots +q^{n-1}X_{i,1}X_{i,2}\cdots X_{i,n-1}))\\
    &=X_{i,0}C_i^{-1}(1+qX_{i,n}+q^2X_{i,1}X_{i,2}+\cdots +q^{n-1}X_{i,1}X_{i,2}\cdots X_{i,n-1})\\
    &=W_{i,n-1}C_i^{-1}.
\end{split}
\end{equation}

We next compute $\cT_i(X_{i,n})$. Firstly one has
\begin{equation}\label{eq:tmui}
    \begin{split}
        \widetilde{\mu}_i(X_{i,n})&=\mu_{i,2}\cdots \mu_{i,n}\cdots\mu_{i,2}(X_{i,n})\\&=\mu_{i,2}\cdots\mu_{i,n}((1+q^{-1}X_{i,n-1})X_{i,n})\\&=(1+q^{-1}\mu_{i,2}\cdots\mu_{i,n}(X_{i,n-1}))\cdot\mu_{i,2}\cdots\mu_{i,n-1}(X_{i,n})^{-1}\\
        &=q(1+(A_{n-1}-A_{n-2})^{-1}(A_{n-2}-1)A_{n-1})(A_{n-1}-1)^{-1}\\
        &=q(A_{n-1}-A_{n-2})^{-1}A_{n-2}
    \end{split}
\end{equation}
The last equality follows from \cref{eq:mukn,eq:muk1k}. 

Hence 
\begin{equation}
    \cT_i(X_{i,n})=X_{i,n-1}^{-1}+\cdots +q^{n-2}X_{i,1}^{-1}X_{i,2}^{-1}\cdots X_{i,n-1}^{-1}=X_{i,0}X_{i,n}W_{i,n-2}C_i^{-1}.
\end{equation}

We next compute $\cT_i(X_{i,1})$. One has
\begin{equation}\label{eq:tmuiXi1}
\begin{split}
    \widetilde{\mu}_i(X_{i,1})&=\mu_{i,2}\cdots\mu_{i,n}\cdots\mu_{i,2}(X_{i,1})\\
    &=\mu_{i,2}\cdots\mu_{i,n}\cdots\mu_{i,3}(X_{i,1}(1+qX_{i,2}^{-1})^{-1})\\&=\mu_{i,2}\cdots\mu_{i,n}\cdots\mu_{i,4}(X_{i,1}(1+qX_{i,3}^{-1}))\cdot (1+q\mu_{i,2}\cdots\mu_{i,n}\cdots\mu_{i,3}(X_{i,2}^{-1}))^{-1}\\&=\cdots\cdots\\&=\mu_{i,2}\cdots\mu_{i,n}(X_{i,1})V_{n-1}\cdots V_2\\&=\mu_{i,2}\cdots\mu_{i,n-1}((1+q^{-1}X_{i,n}^{-1})^{-1}X_{i,1}(1+qX_{i,n}^{-1})^{-1})V_{n-1}\cdots V_2\\
    &=(1+q^{-1}\mu_{i,2}\cdots\mu_{i,n-1}(X_{i,n}^{-1}))^{-1}\cdot\mu_{i,2}\cdots\mu_{i,n-2}((1+q^{-1}X_{i,n-1}^{-1})^{-1}X_{i,1})\\
    &\cdot (1+q\mu_{i,2}\cdots\mu_{i,n-1}(X_{i,n}^{-1}))^{-1}V_{n-1}\cdots V_2\\
    &=(1+q^{-1}\mu_{i,2}\cdots\mu_{i,n-1}(X_{i,n}^{-1}))^{-1}(1+q^{-1}\mu_{i,2}\cdots\mu_{i,n-2}(X_{i,n-1}^{-1}))^{-1}X_{i,1}\\&\cdot (1+q\mu_{i,2}\cdots\mu_{i,n-1}(X_{i,n}^{-1}))^{-1}V_{n-1}\cdots V_2,
\end{split}    
\end{equation}
where 
\begin{equation*}
    V_k=(1+q\mu_{i,2}\cdots\mu_{i,n}\cdots\mu_{i,k+1}(X_{i,k}^{-1}))^{-1},\quad\text{for }2\leq k\leq n.
\end{equation*}

By \cref{eq:mukk,eq:mukn} we have 
\begin{equation*}
    \begin{split}
        1+q^{-1}\mu_{i,2}\cdots\mu_{i,n-1}(X_{i,n}^{-1})=1+(A_{n-1}-1)^{-1}=(A_{n-1}-1)^{-1}A_{n-1},
    \end{split}
\end{equation*}
and
\begin{equation*}
    \begin{split}
        1+q^{-1}\mu_{i,2}\cdots\mu_{i,n-2}(X_{i,n-1}^{-1})&=1+(A_{n-1}-A_{n-2})^{-1}(A_{n-2}-1)\\&=(A_{n-1}-A_{n-2})^{-1}(A_{n-1}-1),
    \end{split}
\end{equation*}

Therefore we have
\begin{equation}\label{eq:aa1}
    (1+q^{-1}\mu_{i,2}\cdots\mu_{i,n-1}(X_{i,n}^{-1}))^{-1}(1+q^{-1}\mu_{i,2}\cdots\mu_{i,n-2}(X_{i,n-1}^{-1}))^{-1}=A_{n-1}^{-1}(A_{n-1}-A_{n-2}).
\end{equation}

We next compute the product $V_{n}\cdots V_2$. By \eqref{eq:muk1k} we have
\begin{equation*}
\begin{split}
    1+q\mu_{i,2}\cdots\mu_{i,n}\cdots\mu_{i,k+1}(X_{i,k}^{-1})&=1+A_k^{-1}(A_{k-1}-1)^{-1}(A_k-A_{k-1})\\&=A_k^{-1}(A_{k-1}-1)^{-1}A_{k-1}(A_k-1),
\end{split}
\end{equation*}
which yields to
\begin{equation*}
    V_k=(A_k-1)^{-1}A_{k-1}^{-1}(A_{k-1}-1)A_k.
\end{equation*}

Then by \eqref{eq:ak1}, we have
\begin{equation*}
\begin{split}
    V_k&=(A_k-1)^{-1}A_{k-1}^{-1}(A_{k-1}+q^{-2}(A_k-A_{k-1}))(A_{k-1}-1)\\&=(A_k-1)^{-1}(1+q^{-2}A_{k-1}^{-1}(A_k-A_{k-1}))(A_{k-1}-1)\\&=(A_k-1)^{-1}(1+q^{-2}(A_k-A_{k-1})(1+q^{-2}(A_{k-1}-1))^{-1})(A_{k-1}-1)\\&=(A_k-1)^{-1}(1+q^{-2}(A_k-1))(1+q^{-2}(A_{k-1}-1))^{-1}(A_{k-1}-1)\\&=((A_k-1)^{-1}+q^{-2})((A_{k-1}-1)^{-1}+q^{-2})^{-1}.
\end{split}
\end{equation*}

Therefore 
\begin{equation}\label{eq:aa2}
    V_{n-1}\cdots V_2=((A_{n-1}-1)^{-1}+q^{-2})(q^{-1}X_{i,n}^{-1}+q^{-2})^{-1}.
\end{equation}

By \eqref{eq:mukn} we have
\begin{equation}\label{eq:aa3}
\begin{split}
    1+q\mu_{i,2}\cdots\mu_{i,n-1}(X_{i,n}^{-1})&=1+q^2(A_{n-1}-1)^{-1}.
\end{split}
\end{equation}

Plug equations \eqref{eq:aa1}, \eqref{eq:aa2} and \eqref{eq:aa3} into \eqref{eq:tmuiXi1}. We get
\begin{equation*}
    \begin{split}
        \widetilde{\mu}_i(X_{i,1})=A_{n-1}^{-1}(A_{n-1}-A_{n-2})X_{i,1}(1+qX_{i,n}^{-1})^{-1}.
    \end{split}
\end{equation*}

Therefore we have
\begin{equation}
    \cT_i(X_{i,1})=W_{i,n-1}^{-1}(W_{i,n-1}-W_{i,n-2})X_{i,n}(1+qX_{i,1}^{-1})^{-1}=q^{-1}W_{i,n-1}^{-1}X_{i,0}^{-1}(1+qX_{i,1}^{-1})^{-1}C_i.
\end{equation}

Next we assume $1<i\leq n$. We firstly compute $\cT_i(X_{i-1,0})$. Notice that in the quiver $\Sigma_n$ the vertex $X_{i-1,0}$ is not connected with vertices $X_{i,t}$ for $2\leq t\leq n$. Hence we have
\begin{equation}
    \cT_i(X_{i-1,0})=M_i(X_{i-1,0})=q^{1/2}X_{i-1,0}X_{i,0}.
\end{equation}

We next compute $\cT_i(X_{i-1,k})$, for $1\leq k\leq n-2$. In this case, one has
\begin{equation}\label{eq:tmui1}
    \begin{split}
        \widetilde{\mu}_i(X_{i-1,k})&=\mu_{i,2}\cdots\mu_{i,n}\cdots\mu_{i,2}(X_{i-1,k})\\
        &=\mu_{i,2}\cdots\mu_{i,n}\cdots\mu_{i,k+1}(X_{i-1,k}(1+qX_{i,k}^{-1})^{-1})\\
        &=X_{i-1,k}(1+q\mu_{i,2}\cdots\mu_{i,k-1}(X_{i,k}))(1+q\mu_{i,2}\cdots\mu_{i,n}\cdots\mu_{i,k+1}(X_{i,k})^{-1})^{-1}\\
        &\overset{(\heartsuit)}{=}X_{i-1,k}(1+(A_{k-1}-1)^{-1}(A_k-A_{k-1}))(1+A_k^{-1}(A_{k-1}-1)^{-1}(A_k-A_{k-1}))^{-1}\\
        &=X_{i-1,k}(A_{k-1}-1)^{-1}(A_k-1)(A_k-1)^{-1}A_{k-1}^{-1}(A_{k-1}-1)A_k\\
        &=X_{i-1,k}A_{k-1}^{-1}A_k.
    \end{split}
\end{equation}
Here $(\heartsuit)$ follows from \eqref{eq:mukk} and \eqref{eq:muk1k}.

It then follows that
\begin{equation}
    \cT_i(X_{i-1,k})=M_i\circ S_i\circ\widetilde{\mu}_i(X_{i-1,k})=X_{i-1,k}W_{i,k-1}^{-1}W_{i,k}=W_{i,k-1}^{-1}X_{i-1,k}W_{i,k}.
\end{equation}
The last equality follows by noticing that $X_{i-1,k}$ commutes with every terms in the summation of $W_{i,k-1}$.

We next compute $\cT_i(X_{i-1,n-1})$. One has 
\begin{equation}\label{eq:tmuin1}
\begin{split}
    \widetilde{\mu}_i(X_{i-1,n-1})&=\mu_{i,2}\cdots\mu_{i,n}\cdots\mu_{i,2}(X_{i-1,n-1})\\&=\mu_{i,2}\cdots\mu_{i,n}\cdots\mu_{i,3}((1+q^{-1}X_{i,2})X_{i-1,n-1})\\&=(1+q^{-1}\mu_{i,2}\cdots\mu_{i,n}\cdots\mu_{i,3}(X_{i,2}))\cdot\mu_{i,2}\cdots\mu_{i,n}(X_{i-1,n-1}(1+qX_{i,n-1}^{-1})^{-1}) \\&=(1+q^{-1}\mu_{i,2}\cdots\mu_{i,n}\cdots\mu_{i,3}(X_{i,2}))\cdot\mu_{i,2}\cdots\mu_{i,n-2}(X_{i-1,n-1}(1+qX_{i,n-1}))\\&\cdot (1+q\mu_{i,2}\cdots\mu_{i,n}(X_{i,n-1}^{-1}))^{-1} \\
    &=(1+q^{-1}\mu_{i,2}\cdots\mu_{i,n}\cdots\mu_{i,3}(X_{i,2})) (1+q^{-1}X_{i,2}^{-1})^{-1} X_{i-1,n-1}\\&\cdot (1+q\mu_{i,2}\cdots\mu_{i,n-2}(X_{i,n-1})) (1+q\mu_{i,2}\cdots\mu_{i,n}(X_{i,n-1}^{-1}))^{-1}.
\end{split}
\end{equation}

Thanks to \cref{eq:mukk,eq:muk1k} we have
\begin{equation*}
    \begin{split}
        1+q^{-1}\mu_{i,2}\cdots\mu_{i,n}\cdots\mu_{i,3}(X_{i,2})&=1+q^{-1}X_{i,2}^{-1}(1+qX_{i,n}+q^2X_{i,n}X_{i,2})\\&=(1+q^{-1}X_{i,n})(1+q^{-1}X_{i,2}^{-1}),
    \end{split}
\end{equation*}
\begin{equation*}
\begin{split}
1+q\mu_{i,2}\cdots\mu_{i,n-2}(X_{i,n-1})&=1+q(X_{i,n-1}^{-1}+\cdots +q^{n-3}X_{i,2}^{-1}\cdots X_{i,n-1}^{-1})^{-1}\\&=1+(A_{n-2}-1)^{-1}(A_{n-1}-A_{n-2})\\&=(A_{n-2}-1)^{-1}(A_{n-1}-1),
\end{split}
\end{equation*}
and
\begin{equation*}
    \begin{split}
        &1+q\mu_{i,2}\cdots\mu_{i,n}(X_{i,n-1}^{-1})\\=&1+q(1+qX_{i,n}+\cdots+q^{n-1}X_{i,n}X_{i,2}\cdots X_{i,n-1})^{-1}(X_{i,n-1}^{-1}+q^{n-3}X_{i,2}\cdots X_{i,n-1}^{-1})^{-1}\\=&1+A_{n-1}^{-1}(A_{n-2}-1)^{-1}(A_{n-1}-A_{n-2})\\=&A_{n-1}^{-1}(A_{n-2}-1)^{-1}A_{n-2}(A_{n-1}-1).
    \end{split}
\end{equation*}

Plug into \eqref{eq:tmuin1}. We have
\begin{equation*}
    \widetilde{\mu}_i(X_{i-1,n-1})=(1+q^{-1}X_{i,n})X_{i-1,n-1}A_{n-2}^{-1}A_{n-1},
\end{equation*}
which yields to
\begin{equation}\label{eq:tii1}
    \cT_i(X_{i-1,n-1})=(1+q^{-1}X_{i,1})X_{i-1,n-1}W_{i,n-2}^{-1}W_{i,n-1}.
\end{equation}

The formula for $\cT_i(X_{i-1,n})$ follows from the identity $X_{i-1,n}=X_{i,1}$ and the previous computation.

Next assume $1\leq i<n$. We compute $\cT_i(X_{i+1,0})$. One has
\begin{equation}\label{eq:tmuii1}
    \begin{split}
        \widetilde{\mu}_i(X_{i+1,0})&=\mu_{i,2}\cdots\mu_{i,n}\cdots\mu_{i,2}(X_{i+1,0})\\&=\mu_{i,2}\cdots\mu_{i,n}(X_{i+1,0})\\&=\mu_{i,2}\cdots\mu_{i,n-1}((1+q^{-1}X_{i,n}^{-1})^{-1}X_{i+1,0})\\&=(1+q^{-1}\mu_{i,2}\cdots\mu_{i,n-1}(X_{i,n}^{-1}))^{-1}(1+q^{-1}\mu_{i,2}\cdots\mu_{i,n-2}(X_{i,n-1}^{-1}))^{-1}X_{i+1,0}\\
        &\overset{(\heartsuit)}{=}(1+(A_{n-1}-1)^{-1})^{-1}(1+(A_{n-1}-A_{n-2})^{-1}(A_{n-2}-1))^{-1}X_{i+1,0}\\
        &=A_{n-1}^{-1}(A_{n-1}-A_{n-2})X_{i+1,0}.
    \end{split}
\end{equation}
The equality $(\heartsuit)$ follows from \eqref{eq:mukn} and \eqref{eq:mukk}.

It then follows that
\begin{equation*}
    \cT_i(X_{i+1,0})=q^{-1/2}W_{i,n-1}^{-1}X_{i+1,0}C_i.
\end{equation*}

To verify the formula for $\cT_i(X_{i+1,k})$ ($2<k\leq n$), we notice that for $k\neq n$ the computation is identical with $\cT_i(X_{i-1,k-1})$, since the two vertices are in the same relative position with respect to mutations involved. The formula for $\cT_i(X_{i+1,n})$ is not considered in the previous computation, but it can be similarly obtained.

The formula for $\cT_i(X_{i+1,k})$, $k=1,2$, follows from the identities $X_{i+1,1}=X_{i,n}$ and $X_{i+1,2}=X_{i-1,n-1}$, and previous computation.

If $X_{t,k}$ is none of the above generators, the corresponding vertex is disconnected from $X_{i,k}$ for $0\leq k\leq n$. Therefore it is fixed by $\cT_i$. We complete the proof.
\end{proof}

Let us mention that in Lemma \ref{le:T} we exclude the case when $n=1$ or $2$, because in these two cases, the quivers are too folded. For example, when $n=1$, the quiver $\Sigma_1$ is disconnected, and when $n=2$, the quiver $\Sigma_2$ has double edges.   

We next prove Theorem \ref{thm:braid}.

\begin{proof}[Proof of Theorem \ref{thm:braid}]

The case when $n=1$ or $2$ can be verified by the straightforward computation, which we will skip. 

Let us assume $n\geq 3$. Take $1\leq i,j\leq n$ with $|i-j|>1$. Suppose firstly $i<j$. By \eqref{eq:relabel} we have
$$X_{j,j-i-1}=X_{i+1,n+i-j+2},\quad X_{j,j-i}=X_{i,n+i-j+1},\quad \text{and }X_{j,j-i+1}=X_{i-1,n+i-j}.$$
Thanks to Lemma \ref{le:T}, we have
\begin{equation}\label{eq:Tij}
    \cT_i(X_{j,t})=\begin{cases}
        W_{i,n+i-j}^{-1}X_{j,j-i-1}W_{i,n+i-j+1}, &\text{if }t=j-i-1,\\
        W_{i,n+i-j+1}^{-1}X_{j,j-i}W_{i,n+i-j-1}, &\text{if }t=j-i,\\
        W_{i,n+i-j-1}^{-1}X_{j,j-i+1}W_{i,n+i-j}, &\text{if }t=j-i+1,\\
        X_{j,t}, &\text{if }t\notin \{j-i-1,j-i,j-i+1\}.
    \end{cases}
\end{equation}
Notice that $W_{i,n+i-j}$ commutes with monomial $X_{j,j-i}X_{j,j-i+1}$. We have
    \begin{equation}\label{eq:Tp}
    \begin{split}
        &\cT_i(X_{j,j-i-1}X_{j,j-i}X_{j,j-i+1})\\=&W_{i,n+i-j}^{-1}X_{j,j-i-1}X_{j,j-i}X_{j,j-i+1}W_{i,n+i-j}=X_{j,j-i-1}X_{j,j-i}X_{j,j-i+1}.
    \end{split}
    \end{equation}
It follows from the definition \eqref{eq:W} that
\begin{equation*}
\begin{split}
    W_{i,n+i-j+1}-W_{i,n+i-j}=&q^{n+i-j+1}X_{i,0}\cdots X_{i,n+i-j+1}\\=&q(W_{i,n+i-j}-W_{i,n+i-j-1})X_{i,n+i-j+1},
\end{split}
\end{equation*}
which yields to
\begin{equation*}
    W_{i,n+i-j+1}+qW_{i,n+i-j-1}X_{j,j-i}=W_{i,n+i-j}(1+qX_{j,j-i}).
\end{equation*}

Therefore 
\begin{equation}\label{eq:Ts}
\begin{split}
    &\cT_i(X_{j,j-i-1}(1+qX_{j,j-i}))\\=&X_{j,j-i-1}W_{i,n+i-j}^{-1}(W_{i,n+i-j+1}+qW_{i,n+i-j-1}X_{j,j-i})\\=&X_{j,j-i-1}(1+qX_{j,j-i}).
\end{split}
\end{equation}

We express $\iota(B_j)$ as following
\begin{equation}\label{eq:Bj}
\begin{split}
    \iota(B_j)&=\sum_{t=0}^{j-i-2}q^tX_{j,0}\cdots X_{j,t}+q^{j-i-1}X_{j,0}\cdots X_{j,j-i-2}(X_{j,j-i-1}(1+qX_{j,j-i}))\\&+q^{j-i+1}X_{j,0}\cdots X_{j,j-i-2}X_{j,j-i-1}X_{j,j-i}X_{j,j-i+1}\\&\cdot(1+qX_{j,j-i+2}+\cdots q^{n-j+i-1}X_{j,j-i+2}\cdots X_{j,n}).
\end{split}
\end{equation}
Thanks to \cref{eq:Tij,eq:Tp,eq:Ts}, each term on the right hand side of \eqref{eq:Bj} is fixed by $\cT_i$. 

Similarly one can show $\iota(k_j)$ is fixed by $\cT_i$. Therefore we deduce that for $|i-j|>1$ and $i<j$ one has
\begin{equation}\label{eq:cTBij}
    \cT_i(\iota(B_j))=B_j=\iota(\bT_i(B_j)),\quad\text{and}\quad \cT_i(\iota(k_j))=k_j=\iota(\bT_i(k_j)).
\end{equation} 

The case when $|i-j|>1$ with $i>j$ can be verified in a similar way. In summary we deduce that \eqref{eq:cTBij} holds whenever $|i-j|>1$.

Suppose $1< i\leq n$, we next compute $\cT_i(\iota(B_{i-1}))$ and $\cT_i(\iota(k_{i-1}))$. Thanks to equation \eqref{eq:T2}, we have 
\begin{equation}\label{eq:tik}
    \cT_i(P_{i-1,k})=q^{1/2}P_{i-1,k}W_{i,k}=\sum_{t=0}^k\ren{P_{i-1,k}P_{i,t}},\quad \text{for $0\leq k\leq n-2$.}
\end{equation}
We also have
\begin{equation}\label{eq:tip}
\begin{split}
    \cT_i(P_{i-1,n-1})=&\cT_i(qP_{i-1,n-2}X_{i-1,n-1})\\=&q^{3/2}P_{i-1,n-2}W_{i,n-2}(1+q^{-1}X_{i,1})X_{i-1,n-1}W_{i,n-2}^{-1}W_{i,n-1}.
\end{split}
\end{equation}

Let us write
\[
W_{i,n-2}=X_{i,0}(1+qX_{i,1}+R),
\]
where
\[
R=q^2X_{i,1}X_{i,2}+\cdots +q^{n-2}X_{i,1}\cdots X_{i,n-2}.
\]

Then
\begin{equation}\label{eq:wxx}
    \begin{split}
        W_{i,n-2}(1+q^{-1}X_{i,1})X_{i-1,n-1}=&X_{i,0}(1+qX_{i,1}+R)(1+q^{-1}X_{i,1})X_{i-1,n-1}\\=&X_{i-1,n-1}X_{i,0}(1+q^3X_{i,1}+R)(1+qX_{i,1})\\=&X_{i-1,n-1}X_{i,0}(1+q^3X_{i,1})(1+qX_{i,1}+R)\\=&X_{i-1,n-1}(1+qX_{i,1})W_{i,n-2}.
    \end{split}
\end{equation}

Plug into \eqref{eq:tip}. We have
\begin{align}
    \cT_i(P_{i-1,n-1})=&q^{3/2}P_{i-1,n-2}X_{i-1,n-1}(1+qX_{i,1})W_{i,n-1}\label{eq:Tipin}\\=&\ren{P_{i-1,n-1}P_{i,0}}+(q+q^{-1})\ren{P_{i-1,n-1}P_{i,1}}\notag\\
    +&\sum_{t=2}^{n-1}\ren{P_{i-1,n-1}P_{i,t}}+\sum_{t=1}^{n-1}\ren{P_{i-1,n}P_{i,t}}.\label{eq:psn}
\end{align}

By \eqref{eq:T2} and \eqref{eq:Tipin}, we deduce that
\begin{equation}\label{eq:tin}
\begin{split}
    \cT_i(P_{i-1,n})=&\cT_i(qP_{i-1,n-1}X_{i-1,n})\\
    =&q^{3/2}P_{i-1,n-2}X_{i-1,n-1}X_{i,1}X_{i,0}^{-1}C_i\\
    =&\ren{P_{i-1,n}P_{i,n}}.
\end{split}
\end{equation}

Therefore by \eqref{eq:tik}, \eqref{eq:psn} and \eqref{eq:tin}, we have 
\begin{equation*}
\begin{split}
     \cT_i(\iota(B_{i-1}))=&\cT_i(\sum_{k=0}^nP_{i-1,k})\\
     =&\sum_{k=0}^{n-2}\sum_{t=0}^k\ren{P_{i-1,k}P_{i,t}}+\ren{P_{i-1,n-1}P_{i,0}}+(q+q^{-1})\ren{P_{i-1,n-1}P_{i,1}}\\
     +&\sum_{t=2}^{n-1}\ren{P_{i-1,n-1}P_{i,t}}+\sum_{t=1}^{n}\ren{P_{i-1,n}P_{i,t}}.
\end{split}
\end{equation*}

For any two elements $f$, $h$ in the quantum torus algebra, let us denote
\begin{equation*}
    \qcom{f}{h}=\frac{q^{1/2}fh-q^{-1/2}hf}{q-q^{-1}}.
\end{equation*}

Then the following equations can be verified directly:
\begin{equation}\label{eq:qcom}
    \qcom{f}{h}=\begin{cases}
        q^{-1/2}fh, &\text{if }fh=qhf,\\
        0, &\text{if }fh=q^{-1}hf,\\
        -q^{3/2}fh, &\text{if }fh=q^{-3}hf,\\
        (q+q^{-1})q^{-3/2}fh, &\text{if }fh=q^3hf.
    \end{cases}
\end{equation}

Recall \eqref{braid}. We have 
\begin{equation*}
    \iota(\bT_i(B_{i-1}))=\qcom{\iota(B_i)}{\iota(B_{i-1})}=\sum_{t=0}^n\sum_{k=0}^n\qcom{P_{i,t}}{P_{i-1,k}}.
\end{equation*}

Recall commuting relations \eqref{eq:com1} and \eqref{eq:com2}. Thanks to \eqref{eq:qcom}, it is now direct to verify  
\begin{equation*}
    \iota(\bT(B_{i-1}))=\cT_i(\iota(B_{i-1})).
\end{equation*}

By \eqref{eq:T2} and \eqref{eq:tin} one has
\begin{equation*}
    \cT_i(\iota(k_{i-1}))=\cT_i(-X_{i-1,0}P_{i-1,n})=-C_{i-1}C_i=\iota(\bT_i(k_{i-1})).
\end{equation*}

Assume $1\leq i<n$. We next verify
\begin{equation*}
    \cT_i(\iota(B_{i+1}))=\iota(\bT_i(B_{i+1})),\quad\text{and}\quad \cT_i(\iota(k_{i+1}))=\iota(\bT_i(k_{i+1})).
\end{equation*}

By \eqref{eq:T3} we have
\begin{equation*}
    \begin{split}
        \cT_i(P_{i+1,0}+P_{i+1,1})=&\cT_i((1+q^{-1}X_{i+1,1})X_{i+1,0})\\
        =&q^{-1/2}(1+q^{-1}X_{i,0}X_{i,n}W_{i,n-2}C_i^{-1})W_{i,n-1}^{-1}X_{i+1,0}C_i\\
        =&q^{-3/2}X_{i,0}X_{i,n}(P_{i,n-1}+W_{i,n-2})W_{i,n-1}^{-1}X_{i+1,0}\\
        =&\ren{P_{i,0}P_{i+1,1}}.
    \end{split}
\end{equation*}

Still by \eqref{eq:T3} we have
\begin{align}
        \cT_i(P_{i+1,2})=&\cT_i(X_{i+1,1}X_{i+1,2}X_{i+1,0})\notag\\
        =&q^{-1/2}X_{i,0}X_{i+1,1}W_{i,n-2}C_i^{-1}(1+q^{-1}X_{i,1})X_{i+1,2}W_{i,n-2}^{-1}X_{i+1,0}C_i\notag\\
        =&q^{-1/2}X_{i,0}X_{i+1,1}X_{i+1,2}(1+qX_{i,1})X_{i+1,0}\label{eq:tp1}\\
        =&\ren{P_{i,0}P_{i+1,2}}+\ren{P_{i,1}P_{i+1,2}}.\label{eq:tp2}
\end{align}
Here the third equality follows from \eqref{eq:wxx}.

Then by \eqref{eq:T3} and \eqref{eq:tp1}, for $3\leq k\leq n$ one has
\begin{equation}\label{eq:tp3}
    \begin{split}
        \cT_i(P_{i+1,k})=&\cT_i(q^{k-2}P_{i+2,2}X_{i+1,3}\cdots X_{i+1,k})\\
        =&q^{k-5/2}X_{i,0}X_{i+1,1}X_{i+1,2}X_{i+1,0}(1+qX_{i,1})W_{i,1}^{-1}X_{i+1,2}\cdots X_{i+1,k}W_{i,k-1}\\
        =&\sum_{t=0}^{k-1}\ren{P_{i,t}P_{i+1,k}}.
    \end{split}
\end{equation}

Therefore we have
\begin{equation*}
\cT_i(\iota(B_{i+1}))=\cT_i(\sum_{k=0}^nP_{i+1,k})=\sum_{k=1}^n\sum_{t=0}^{k-1}\ren{P_{i,t}P_{i+1,k}}.
\end{equation*}

On the other hand, we have
\begin{equation*}
    \iota(\bT_i(B_{i+1}))=\qcom{\iota(B_i)}{\iota(B_{i+1})}=\sum_{t=0}^n\sum_{k=0}^n\qcom{P_{i,t}}{P_{i+1,k}}.
\end{equation*}
Here each term in the summation can be computed by \eqref{eq:qcom} and commuting relations \eqref{eq:com1} and \eqref{eq:com2}. It is now direct to verify 
\[
\cT_i(\iota(B_{i+1}))=\iota(\bT_i(B_{i+1}).
\]

By \eqref{eq:T3} and \eqref{eq:tp3} we have 
\[
\cT_i(\iota(k_{i+1}))=\cT_i(P_{i+1,n}X_{i+1,0})=-C_iC_{i+1}=\iota(\bT(k_{i+1})).
\]

Assume $1\leq i\leq n$. Finally we need to verify
\begin{equation*}
    \cT_i(\iota(B_i))=\iota(\bT_i(B_i)),\quad\text{and}\quad \cT_i(\iota(k_i))=\iota(\bT_i(k_i)).
\end{equation*}

For $1\leq k\leq n-1$, we claim 
\begin{equation}\label{claim}
    \cT_i(P_{i,k})=W_{i,k-1}^{-1}-W_{i,k}^{-1}.
\end{equation}

We prove the claim by induction on $k$. Firstly, by \eqref{eq:T1} we have
\begin{equation*}
    \cT_i(P_{i,1})=\cT_i(qX_{i,0}X_{i,1})=X_{i,0}^{-1}(1+qX_{i,1}^{-1})^{-1}=W_{i,0}^{-1}-W_{i,1}^{-1}.
\end{equation*}

Suppose the equation \eqref{claim} holds for $k$ ($1<k<n-1$). Then by \eqref{eq:T1} we have
\begin{equation*}
\begin{split}
\cT_i(P_{i,k+1})=&\cT_i(q^{-1}X_{i,k+1}P_{i,k})\\
=&q^{-1}W_{i,k+1}^{-1}X_{i,k+1}W_{i,k-1}(W_{i,k-1}^{-1}-W_{i,k}^{-1})\\
=&q^{-1}W_{i,k+1}^{-1}X_{i,k+1}(W_{i,k}-W_{i,k-1})W_{i,k}^{-1}\\
=&W_{i,k+1}^{-1}(W_{i,k+1}-W_{i,k})W_{i,k}^{-1}\\
=&W_{i,k}^{-1}-W_{i,k+1}^{-1}.
\end{split}
\end{equation*}
Hence the claim is proved.

Then we have
\begin{equation*}
\begin{split}
    \cT_i(P_{i,n})=&\cT_i(q^{-1}X_{i,n}P_{i,n-1})\\=&q^{-1}X_{i,0}X_{i,n}W_{i,n-2}C_i^{-1}(W_{i,n-2}^{-1}-W_{i,n-1}^{-1})\\
    =&q^{-1}X_{i,0}X_{i,n}(1-(W_{i,n-1}-qX_{i,n}^{-1}X_{i,0}^{-1}C_i)W_{i,n-1}^{-1})C_i^{-1}\\
    =&W_{i,n-1}^{-1}.
\end{split}
\end{equation*}

Therefore
\begin{equation*}
    \begin{split}
        \cT_i(\iota(B_{i}))=&\cT_i(\sum_{k=0}^nP_{i,k})
        =W_{i,n-1}C_i^{-1}+X_{i,0}^{-1}
        =W_{i,n}C_i^{-1}=\iota(\bT(B_i)),
    \end{split}
\end{equation*}
and
\begin{equation*}
    \cT_i(\iota(k_i))=-\cT_i(X_{i,0}P_{i,n})=C_i^{-1}=\iota(\bT(k_i)).
\end{equation*}

We complete the proof.
\end{proof}

One consequence of Theorem \ref{thm:braid} is that the integral forms of $\imath$quantum groups, considered in Section \ref{sec:3}, are compatible with our cluster realisations.

Recall the algebra embedding $\iota$ in Theorem \ref{thm:emb}, and the integral form $\ZtUi$ in Definition \ref{def:int}.

\begin{cor}\label{cor:emq}
Under the embedding $\iota$, we have
\begin{equation}
    \ZtUi=\tUi\cap \mathcal{O}_q(\mathcal{X}_{|\Sigma_n|}).
\end{equation}
Therefore $\iota$ restricts to an $\A$-algebra embedding $\ZtUi\hookrightarrow \mathcal{O}_q(\mathcal{X}_{|\Sigma_n|})$. 
\end{cor}

\begin{proof}
    By Proposition \ref{prop:qcl}, the intersection $\tUi\cap \mathcal{O}_q(\mathcal{X}_{|\Sigma_n|})$ contains elements $B_i$, $k_i$, for $1\leq i\leq n$. By Theorem \ref{thm:braid}, the intersection is invariant under braid group symmetries $\bT_i$, for $1\leq i\leq n$. Thanks to Lemma \ref{le:inti}, we have $\ZtUi\subseteq \tUi\cap \mathcal{O}_q(\mathcal{X}_{|\Sigma_n|})$. To show the inclusion of another direction, we take the reduced expression $\mathbf{i}$ as in Proposition \ref{prop:inj}. Take any $x\in \tUi\cap \mathcal{O}_q(\mathcal{X}_{|\Sigma_n|})$. Expand $x$ in terms of PBW-basis (of $\imath$quantum groups) associated with $\mathbf{i}$ (see \eqref{eq:PBWb}), say $x=\sum_bc_bb$. It follows from the proof of Proposition \ref{prop:inj}, that the coefficient of the leading monomial of $\iota(b)$ belongs to $\pm q^{\mathbb{Z}}$. Since by the assumption, $\iota(x)$ belongs to the quantum torus algebra $\T_{\Sigma_n}$, we deduce that each coefficient $c_b$ belongs to $\A$. This completes the proof.
\end{proof}

Recall the central element $C_i$ in \eqref{eq:cen} and the integral form $\ZdUi$ in \eqref{eq:intd}.

\begin{cor}\label{cor:cen}
    The map $\iota$ descends to an $\A$-algebra embedding
    \begin{equation}
        \ZdUi\hookrightarrow \mathcal{O}_q(\mathcal{X}_{|\Sigma_n|})/(C_i-1,1\leq i\leq n).
    \end{equation}
\end{cor}

\subsection{Cyclic symmetries}\label{sec:rho}

Recall in Section \ref{sec:ace} that the braid group action associated with a Coxeter element gives a cyclic symmetry of the $\imath$quantum group. In this subsection, we show that, based on our cluster realisation, this cyclic symmetry corresponds to rotating the quiver.

Let us add a new vertex $X_0$ to the quiver $\Sigma_n$, add dashed arrows $X_{n,0}\dasharrow X_0$, $X_0\dasharrow X_{1,0}$, and add arrows $X_{1,1}\rightarrow X_0$, $X_0\rightarrow X_{n,n}$ of weight one. The resulting quiver is denoted by $\Sigma_n'$.

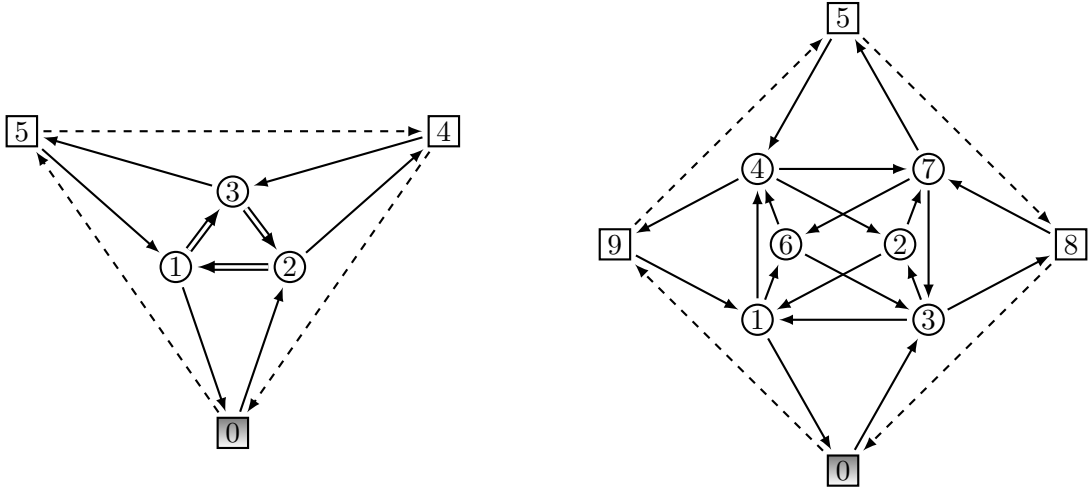
\begin{figure}[h]
    \centering
    \begin{tikzpicture}[every node/.style={inner sep=0, minimum size=0.4cm, thick, draw, fill=white}, thick, x=0.75cm, y=1cm]
     \node (1) at (-14.4,1.5) [draw] {$\tiny{5}$};
     \node (4) at (-7,1.5) [draw] {$\tiny{4}$};
     \node (2) at (-11.7,-0.3) [draw, circle] {$\tiny{1}$};
     \node (3) at (-10.7,0.7) [draw, circle] {$\tiny{3}$};
     \node (5) at (-9.7,-0.3) [draw, circle] {$\tiny{2}$};
     \node (0) at (-10.7,-2.5) [draw, shade] {$\tiny{0}$};

     \begin{scope}[>=latex]
     \qdarrow{1}{4};
     \qdarrow{4}{0};
     \qdarrow{0}{1};
     \qddarrow{3}{5};
     \qddarrow{5}{2};
     \qddarrow{2}{3};
     \qarrow{4}{3};
     \qarrow{3}{1};
     \qarrow{1}{2};
     \qarrow{2}{0};
     \qarrow{0}{5};
     \qarrow{5}{4};

     \end{scope}
    
        \node (1) at (-4,0) [draw] {$\tiny{9}$};
        \node (2) at (0,3) [draw] {$\tiny{5}$};
        \node (3) at (4,0) [draw] {$\tiny{8}$};
        \node (4) at (0,-3) [draw,shade] {$\tiny{0}$};
        \node (5) at (-1.5,1) [draw, circle] {$\tiny{4}$};
        \node (6) at (1.5,1) [draw,circle] {$\tiny{7}$};
        \node (7) at (-1.5,-1) [draw,circle] {$\tiny{1}$};
        \node (8) at (1.5,-1) [draw,circle] {$\tiny{3}$};
        \node (9) at (1,0) [draw, circle] {$\tiny{2}$};
        \node (10) at (-1,0) [draw, circle] {$\tiny{6}$};

        \begin{scope}[>=latex]
            \qdarrow{2}{3};
            \qdarrow{3}{4};
            \qdarrow{4}{1};
            \qdarrow{1}{2};
            \qarrow{1}{7};
            \qarrow{7}{4};
            \qarrow{4}{8};
            \qarrow{8}{3};
            \qarrow{3}{6};
            \qarrow{6}{2};
            \qarrow{2}{5};
            \qarrow{5}{1};
            \qarrow{5}{6};
            \qarrow{6}{8};
            \qarrow{8}{7};
            \qarrow{7}{5};
            \qarrow{5}{9};
            \qarrow{9}{6};
            \qarrow{6}{10};
            \qarrow{10}{5};
            \qarrow{7}{10};
            \qarrow{10}{8};
            \qarrow{8}{9};
            \qarrow{9}{7};
        \end{scope}
    \end{tikzpicture}
    \caption{$\Sigma_2'$-quiver (left) and $\Sigma_3'$-quiver (right)}
    \label{fig:sp}
\end{figure}

It is direct to see that there is an automorphism $\rho_n$ of the quiver $\Sigma_n'$, given by $\rho_n(X_{i,j})=X_{i+1,j}$ for $1\leq j\leq i<n$, $\rho_n(X_{n,j})=X_{n-j+1,n-j+1}$ for $j\neq 0$, $\rho_n(X_{i,0})=X_{i+1,0}$ for $1\leq i<n$, $\rho_n(X_{n,0})=X_0$ and $\rho_n(X_0)=X_{1,0}$.

One can put the quiver $\Sigma_n'$ in a symmetric way, where the map $\rho_n$ can be visualized as rotating the quiver $\Sigma_n'$ by degree $\frac{2\pi}{n+1}$ (except for the vertices $X_{t,(n+1)/2}$, for $(n+1)/2\leq t\leq n$, when $n$ is odd, for which vertices we rotate $\frac{4\pi}{n+1}$ degree). In particular we see that $\rho_n$ has order $n+1$. 

\begin{example}
    Follow the notations of Figure \ref{fig:sp}. For $n=2$, in the $\Sigma_2'$-quiver, the map $\rho_2$ reads
    \begin{equation*}
        \begin{array}{ccc}
            X_5\mapsto X_4,\quad & X_4
            \mapsto X_0, \quad & X_0\mapsto X_5,\\
          X_1\mapsto X_3,\quad  & X_3\mapsto X_2, \quad & X_2\mapsto X_1.
        \end{array}
    \end{equation*}
    For $n=3$, in the $\Sigma_3'$-quiver, the map $\rho_3$ reads
    \begin{equation*}
        \begin{array}{ccccc}
            X_9\mapsto X_5,\quad & X_5\mapsto X_8,\quad & X_8\mapsto X_0, \quad & X_0\mapsto X_9, \quad & X_1\mapsto X_4, \\
            X_4\mapsto X_7 \quad & X_7\mapsto X_3, \quad & X_3\mapsto X_1,\quad & X_2 \mapsto X_6, \quad & X_6\mapsto X_2.
        \end{array}
    \end{equation*}
\end{example}

Set 
\begin{equation*}
M_n'=\ren{\prod_{X\in v(\Sigma_n')}X}\qquad\text{and}\qquad  M_n=\ren{\prod_{X\in v(\Sigma_n)}X}
\end{equation*}
to be the normalized monomials in the quantum torus algebras $\T_{\Sigma'_n}$ and $\T_{\Sigma_n}$, respectively. Here the products are taken over all the generators of the corresponding quantum torus algebras.

It is easy to verify that $M_n'$ is a central element in $\T_{\Sigma_n'}$. There is a surjective algebra homomorphism
\begin{equation*}
    \T_{\Sigma_n'}\longrightarrow\T_{\Sigma_n},
\end{equation*}
given by $X\mapsto X$ for $X\neq X_0$, and $X_0\mapsto M_n^{-1}$. The kernel of this map is the two-sided ideal of $\T_{\Sigma_n'}$ generated by $M_n'-1$.

The cyclic symmetry $\rho_n$ induces an algebra automorphism of $\T_{\Sigma_n'}$, which clearly preserves the element $M_n'-1$. Therefore it descends to an algebra automorphism on $\T_{\Sigma_n}$, which we still denote by $\rho_n$.

Recall the automorphism $\bT_c$ of $\tUi$ in Section \ref{sec:ace}.

\begin{prop}\label{prop:rho}
    The map $\rho_n$ restricts to the subalgebra $\tUi$ under the embedding $\iota$. Moreover it coincides with the automorphism $\bT_c$ on $\tUi$. 
\end{prop}

\begin{proof}
    It suffices to show that maps $\rho_n$ and $\bT_c$ coincide on the generators $B_i$, $k_i$, for $1\leq i\leq n$. We firstly compute $\rho_n(B_i)$ and $\rho_n(k_i)$.

For $1\leq i<n$, we have
\begin{equation}\label{eq:rh1}
    \rho_n(B_i)=\rho_n(\sum_{t=0}^n\ren{X_{i,0}X_{i,1}\cdots X_{i,t}})=\sum_{t=0}^n\ren{X_{i+1,0}X_{i+1,1}\cdots X_{i+1,t}}=B_{i+1},
\end{equation}
and 
\begin{equation}\label{eq:rh2}
    \rho_n(k_i)=-\rho_n(\ren{X_{i,0}^2X_{i,1}\cdots X_{i,n}})=-\ren{X_{i+1,0}^2X_{i+1,1}\cdots X_{i+1,n}}=k_{i+1}.
\end{equation}

We also have
\begin{equation}\label{eq:rh4}
    \rho_n(B_n)=\rho_n(\sum_{t=0}^n\ren{X_{n,0}X_{n,1}\cdots X_{n,t}})=M_n^{-1}+\sum_{t=1}^n\ren{M_n^{-1}X_{n,n}X_{n-1,n-1}\cdots X_{n-k+1,n-k+1}},
\end{equation}
and
\begin{equation}\label{eq:rh3}
    \rho_n(k_n)=-\rho_n(\ren{X_{n,0}^2X_{n,1}\cdots X_{n,n}})=-\ren{M_n^{-2}X_{1,1}X_{2,2}\cdots X_{n,n}}.
\end{equation}

We next consider $\bT_c(B_i)$ and $\bT_c(k_i)$, for $1\leq i\leq n$.

For $1\leq i<n$, we have $\bT_c(B_i)=B_{i+1}$ by the proof of Proposition \ref{prop:tci}.

We also have
\begin{equation*}
    \bT_1\bT_2\cdots\bT_n(k_i)=\bT_1\cdots\bT_{i+1}(k_i)=\bT_1\cdots\bT_i(-k_ik_{i+1})=\bT_{1}\cdots\bT_{i-1}(k_{i+1})=k_{i+1}.
\end{equation*}

Comparing with \eqref{eq:rh1} and \eqref{eq:rh2}, we deduce that $\rho_n(B_i)=\bT_c(B_i)$ and $\rho_n(k_i)=\bT_c(B_i)$, for $1\leq i<n$.

It remains to compute $\bT_1\bT_2\cdots\bT_n(B_n)$ and $\bT_1\bT_2\cdots\bT_n(k_n)$. It follows from easy computation that
\begin{equation*}
    \bT_c(k_n)=\bT_1\cdots\bT_{n-1}(k_n^{-1})=\bT_1\cdots\bT_{n-2}(-k_{n-1}^{-1}k_{n}^{-1})=\cdots=(-1)^{n-1}k_1^{-1}k_2^{-1}\cdots k_n^{-1}.
\end{equation*}

Recall the monomial $P_{i,j}$ in \eqref{eq:pik}. Then under the embedding $\iota$, in $\T_{\Sigma_n}$ we have
\begin{equation}\label{eq:Tk}
    \bT_c(k_n)=-\ren{P_{1,0}^{-2}P_{2,1}^{-2}\cdots P_{n,n-1}^{-2}X_{1,1}^{-1}X_{2,2}^{-1}\cdots X_{n,n}^{-1}}=-\ren{M_n^{-2}X_{1,1}X_{2,2}\cdots X_{n,n}}.
\end{equation}
Comparing with \eqref{eq:rh3}, we deduce that $\rho_n(k_n)=\bT_c(k_n)$.

Finally, we compute the image of $\bT_c(B_n)$ in under $\iota$. By the similar argument as in the proof of Proposition \ref{prop:inj}, with commuting relations \eqref{eq:com1} and \eqref{eq:com2}, one can verify the following formula by induction on $k$,
    \begin{equation*}
        \bT_1\bT_2\cdots \bT_{k-1}(B_k)=\sum_{t_1<t_2<\cdots<t_k}\ren{P_{1,t_1}P_{2,t_2}\cdots P_{k,t_k}},\quad\text{for }1\leq k\leq n.
    \end{equation*}

    In particular, we have
    \begin{equation}\label{eq:TB}
    \begin{split}
        \bT_1\bT_2\cdots\bT_{n-1}(B_n)&=\ren{P_{1,0}P_{2,1}\cdots P_{n,n-1}}\\
    &+\sum_{k=1}^n\ren{P_{1,0}P_{2,1}\cdots P_{n,n-1}X_{n,n}X_{n-1,n-1}\cdots X_{n-k+1,n-k+1}}\\
    &=\ren{M_nX_{1,1}^{-1}\cdots X_{n,n}^{-1}}+\sum_{k=1}^n\ren{M_nX_{1,1}^{-1}\cdots X_{n-k.n-k}^{-1}}.
    \end{split}
    \end{equation}

    By \eqref{eq:Tk} and \eqref{eq:TB}, we have
    \begin{equation*}
    \begin{split}
        \bT_c(B_n)&=\bT_1\cdots\bT_{n-1}(-k_n^{-1}B_n)\\
        &=\ren{M_n^{-1}}+\sum_{k=1}^n\ren{M_n^{-1}X_{n,n}X_{n-1,n-1}\cdots X_{n-k+1,n-k+1}},
    \end{split}
    \end{equation*}
    which equals $\rho_n(B_n)$ by \eqref{eq:rh4}. We complete the proof.
\end{proof}

\bibliographystyle{amsplain}

\end{document}